\documentclass[12pt,twoside,reqno]{amsart}
\usepackage[latin1]{inputenc}
\usepackage[english]{babel}
\usepackage{amsmath, amsthm, amsfonts,amssymb}
\usepackage{color}
\usepackage{hyperref}

\usepackage[left=2.5cm,right=2.5cm,top=2cm,bottom=2cm]{geometry}

\usepackage{anysize}
\marginsize{10cm}{10cm}{2cm}{2cm}
\newtheorem{thm}{Theorem}[section]
\newtheorem{cor}[thm]{Corollary}
\newtheorem{lem}[thm]{Lemma}
\newtheorem{prop}[thm]{Proposition}

\theoremstyle{plain}

\theoremstyle{definition}

\newtheorem{defn}[thm]{Definition}
\newtheorem{rem}[thm]{Remark}
\newtheorem{ex}[thm]{Example}

\newcommand\blfootnote[1]{%
	\begingroup
	\renewcommand\thefootnote{}\footnote{#1}%
	\addtocounter{footnote}{-1}%
	\endgroup
}

\newcommand{\N}{\mathbb{N}}

\newcommand{\X}{\mathbb{X}}

\newcommand{\supp}{\mathrm{supp \,}}
\newcommand{\sgn}{\mathrm{sgn \,}}

\def\ba{\begin{eqnarray*}}
	\def\ea{\end{eqnarray*}}

\def\bee{\begin{equation}}

\def\ene{\end{equation}}

\title{The weighted Property (A) and the greedy algorithm}

\author{P. M. Bern\'a}
\address{Pablo M. Bern\'a
	\\
	Departmento de Matem\'aticas
	\\
	Universidad Aut\'onoma de Madrid
	\\
	28049 Madrid, Spain} \email{pablo.berna@uam.es}

\author{S. J. Dilworth}
\address{Stephen J. Dilworth
	\\
	Department of Mathematics
	\\
	University of South Carolina
	\\
	Columbia SC 29208, USA} \email{dilworth@math.sc.edu}

\author{D. Kutzarova}
\address{Denka Kutzarova
	\\
	Department of Mathematics
	\\
	University of Illinois Urbana-Champaign
	\\
	Urbana, IL 61801, USA; and Institute of Mathematics and Informatics, Bulgarian Academy of Sciences} \email{denka@math.uiuc.edu}

\author{T. Oikhberg}
\address{Timur Oikhberg
	\\
	Department of Mathematics
	\\
	University of Illinois Urbana-Champaign
	\\
	Urbana, IL 61801, USA} \email{oikhberg@illinois.edu}

\author{B. Wallis}
\address{Ben Wallis
	\\
	Department of Mathematical Sciences
	\\
	Northern Illinois University
	\\
	DeKalb, IL 60115-2888, USA} \email{benwallis@live.com}

\date{}
\begin{document}
	\maketitle

	\begin{abstract}
		We investigate various aspects of the ``weighted'' greedy algorithm with respect to a Schauder basis.
		For a weight $w$, we describe $w$-greedy, $w$-almost-greedy, and $w$-partially-greedy bases,
		and examine some properties of $w$-semi-greedy bases. To achieve these goals, we introduce
		and study the $w$-Property (A).
		% In this paper we introduce the concept of the $w$-Property (A)
		% and we show that a basis is $w$-greedy if and only if the basis is unconditional ans satisfies
		% the $w$-Property (A), recovering the constant one.
		% Also, we study some properties of the $w$-Property (A)
		% improving some results that we can find in \cite{DKTW}
		% and we introduce the notion of $w$-partially-greedy bases.
	\end{abstract}
	\blfootnote{\hspace{-0.031\textwidth} 2000 Mathematics Subject Classification. 46B15, 41A65.\newline
		The first author was supported by a PhD fellowship FPI-UAM and the grants MTM-2016-76566-P (MINECO, Spain) and 19368/PI/14 (\emph{Fundaci\'on S\'eneca}, Regi\'on de Murcia, Spain). 
		The second author was supported by the National Science Foundation under Grant Number DMS--1361461. The second and third authors were supported by the Workshop in Analysis and Probability at Texas A\&M University in 2017. }
	
	% % % % % % % % % % % % % % % % % % % %INTRODUCCIÓN
	
% 	\begin{section}{Introduction}

\section{Introduction}\label{s:intro}

		In this paper, we investigate the operation of the ``weighted'' greedy algorithm, and
		its efficiency. Throughout, $(\X,\Vert \cdot \Vert)$ is a real Banach space with
		a semi-normalized Schauder basis $\mathcal{B}=(e_n)_{n=1}^{\infty}$, with biorthogonal functionals
		$(e_n^{*})_{n=1}^{\infty}$; that is,
		\begin{enumerate}
			\item[A1)] $0<c_1:= \inf_n \min \lbrace\Vert e_n\Vert,\Vert e_n^*\Vert\rbrace\leq \sup_n \max \lbrace\Vert e_n\Vert,\Vert e_n^*\Vert\rbrace=:c_2<\infty$,
			\item[A2)] $e_i^*(e_j)= 1$ if $i=j$ and $e_i^*(e_j)=0$ for $i\neq j$,
			\item[A3)] $\mathbb{X}=\overline{span[e_i : i\in\mathbb N]}$,
			\item[A4)] $\|S_m\| \leq K$ for every $m$, where $(S_m)_m$ are partial sum operators 
%  $\left\Vert S_m\left(\sum_{i=1}^\infty a_i e_i\right)\right\Vert \leq K\Vert \sum_{i=1}^\infty a_ie_i\Vert,$
% 			where $(S_m)_{m=1}^\infty$ are partial sums -- that is, % is the algorithm of the partial sums,
-- that is, $S_m(\sum_{i=1}^\infty a_i e_i)= \sum_{i=1}^m a_ie_i$. We denote by $K_b$ % the least constant that satisfies the last inequality and it's called 
the least value of $K$ for which the preceding inequality holds, and call it the \textbf{basis constant}.
		\end{enumerate}

		We will refer to $\mathcal B$ as a \textbf{basis}. Of course, for every $x\in\mathbb X$, there exists a unique expansion $x=\sum_j e_j^*(x)e_j$. As usual, $\supp(x)=\{i\in \N: e_i^*(x)\ne 0\}$, $|A|$ denotes the cardinality of a set $A$ and $$\mathbb{N}^m = \lbrace A\subset \mathbb N : \vert A\vert= m\rbrace,\;\; \mathbb N^{<\infty} = \bigcup_{m=0}^\infty\mathbb N^{m}.$$
		Further notations will be often used: 
		if $a$ and $b$ are functions of some variable, $a\lesssim b$ means that there exists a constant $c>0$
		such that $a\leq c\cdot b$; if $A$ and $B$ are subsets of $\N$, $A<B$ means that
		$\max_{j\in A}j < \min_{j\in B} j$, $P_A$ is the projection operator, i.e, if $A$ is a finite set, $P_A(\sum_j a_j e_j) = \sum_{j\in A}a_je_j$ and
$P_A^c=I-P_A$ is the complementary projection, $\mathbf{1}_{\varepsilon A}=\sum_{n\in A}\varepsilon_n e_n$ for
$\varepsilon_n \in \lbrace \pm 1\rbrace$ and if $\varepsilon_n \equiv 1$, we write $\mathbf{1}_A$.
		
		In 1999, S. V. Konyagin and V. N. Temlyakov introduced in \cite{KT} the
		\textbf{Thresholding Greedy Algorithm (TGA)}: for in $x \in \mathbb X$
		we produce the sequence of \textbf{greedy approximands} % $(\mathcal G_m)_{m=1}^\infty$;
		$$\mathcal G_m(x) = \sum_{n=1}^me_{\pi(n)}^{*}(x)e_{\pi(n)},$$ where $\pi$ is a \textbf{greedy ordering},
		that is, $\pi: \lbrace 1, 2,..., \vert \supp x \vert\rbrace \longrightarrow \supp x$ is a bijection such that
		% $\supp(x)\subseteq \pi(\mathbb{N})$ and 
		$\vert e^*_{\pi(i)}(x)\vert \geq \vert e^*_{\pi(j)}(x)\vert$ for $i\leq j$. 
		Alternatively we can write $\mathcal G_m(x) = \sum_{k \in A_m(x)} e_k^*(x) e_k$, where $A_m(x) = \{ \pi(n) : n \leq m\}$
		is a \textbf{greedy set} of $x$: $\inf_{k \in A_m(x)} \vert e_k^*(x) \vert \geq \sup_{k \notin A_m(x)} \vert e_k^*(x) \vert$.
		
		Also, they defined in \cite{KT} the \textbf{quasi-greedy bases} as those bases such that there exists a positive constant $C$ such that
		\begin{eqnarray}\label{quasi}
		\Vert \mathcal{G}_m(x)\Vert \leq C\Vert x\Vert,\; \forall x\in\mathbb X, \forall m \in \mathbb N.
		\end{eqnarray}
		P. Wojtaszczyk proved in \cite{Woj} that a basis is quasi-greedy if and only if the (TGA) converges -- that is,
		$$\lim_{m\rightarrow \infty}\Vert x- \mathcal G_m(x)\Vert = 0,\; \forall x\in\mathbb X.$$
		Of course, \eqref{quasi} is equivalent to the existence of a constant $C'$ such that
		\begin{eqnarray}\label{quasi2}
		\Vert x-\mathcal{G}_m(x)\Vert \leq C'\Vert x\Vert,\; \forall x\in\mathbb X,\forall m\in \mathbb N.
		\end{eqnarray}
		We denoted by $C_q$ the least constant that satisfies \eqref{quasi2}, it is called the quasi-greedy constant and we say that $\mathcal B$ is $C_q$-quasi-greedy.
		
		On the other hand, the (TGA) is a good candidate to obtain the best m-term approximation with regard to $\mathcal B$. In this sense, S. V. Konyagin and V. N. Temlyakov defined in \cite{KT} the \textbf{greedy bases} as those bases such that there exists a constant $C\geq 1$ such that
		\begin{eqnarray}
		\Vert x-\mathcal{G}_m(x)\Vert \leq C\inf\lbrace \Vert x-\sum_{n\in A}a_ne_n\Vert : A\subset \mathbb N, \vert A\vert = m, a_n\in\mathbb R\rbrace.
		\end{eqnarray}
		Furthermore, they showed that $\mathcal B$ is greedy if and only if $\mathcal B$ is democratic (that is, $\Vert \mathbf{1}_A\Vert \lesssim \Vert \mathbf{1}_B\Vert$, for all $\vert A\vert \leq \vert B\vert$) and unconditional.
		
		Some years later, G. Kerkyacharian, D. Picard and V. N. Temlyakov \cite{GPT} introduced the following extension of the greedy bases: we consider a weight $w=(w_i)_{i=1}^\infty\in (0,\infty)^{\mathbb N}$. If $A\subset \mathbb N$, $w(A)=\sum_{i\in A}w_i$ denote the $w$-measure of $A$. We define the error $\sigma_\delta^w(x)$ as $$\sigma_\delta^w(x,\mathcal B)_{\mathbb X}=\sigma_\delta^w(x):=\inf\lbrace \Vert x-\sum_{n\in A}a_n e_n\Vert : A\in\mathbb N^{<\infty}, w(A)\leq\delta, a_n\in\mathbb R\rbrace.$$
		
		\begin{defn}
			We say that $\mathcal B$ is \textbf{$w$-greedy} if there exists a constant $C\geq 1$ such that
			\begin{eqnarray}\label{wgreedy}
			\Vert x-\mathcal{G}_m(x)\Vert \leq C\sigma_{w(A_m(x))}^w(x),\; \forall x\in\mathbb X, \forall m\in\mathbb N.
			\end{eqnarray}
			We denote by $C_g$ the least constant that satisfies \eqref{wgreedy} and we say that $\mathcal B$ is $C_g$-$w$-greedy.
		\end{defn}
		
		Roughly, the greedy bases are those where the greedy approximation is ``as effective as
		$m$-term approximation can possibly be''.
		
		This generalization was motivated by the work of A. Cohen,
		R. A. DeVore and R. Hochmuth in \cite{CDH}.
		In their recent paper \cite{BB2}, the first author and \'O. Blasco characterize $w$-greedy bases
		using the best $m$-term error in the approximation ``with polynomials of constant coefficients''.
		% and recently, P.M. Bern\'a and \'O. Blasco characterize in \cite{BB2} these bases using the best $m$-term error
		% in the approximation with respect to polynomials with constant coefficients in the context
		% of the weak greedy algorithm and weights.
		Moreover,
		% in \cite{Tem}, we can find the characterization of $w$-greedy bases in terms of
		% the $w$-democracy and the unconditionality. 
		\cite{Tem} characterizes $w$-greedy bases in terms of their $w$-democracy and unconditionality. 
		
		\begin{defn}
			We say that $\mathcal B$ is \textbf{$w$-democratic} if there exists a constant $C\geq 1$ such that
			\begin{eqnarray}\label{wdem}
			\Vert \mathbf{1}_A\Vert \leq C\Vert \mathbf{1}_B\Vert,
			\end{eqnarray}
			for any pair of sets $A,B\in\mathbb N^{<\infty}$ with $w(A)\leq w(B)$.
			We denote by $C_d$ the least constant that satisfies \eqref{wdem} and we say that $\mathcal B$ is $C_d$-$w$-democratic.
		\end{defn}
		
		Recall that a basis $\mathcal B$ in $\mathbb X$ is % (\textbf{suppression})
		\textbf{unconditional} if any rearrangement of $\sum_{n}e_n^*(x)e_n$
		converges in norm to $x$ for any $x\in\mathbb X$.
		% This notion is equivalent to the fact
		% that the projection operator is uniformly bounded, i.e, 
		This is equivalent to the uniform boundedness of basis projections:
		\begin{eqnarray}\label{uncon}
		\Vert x-P_A(x)\Vert \leq K\Vert x\Vert, \; \forall x\in\mathbb X, \forall A\subset \mathbb N.
		\end{eqnarray}
		We denote by $K_u$ the least constant that satisfies \eqref{uncon}, it is called the (suppression) unconditional constant,
and we say that $\mathcal B$ is $K_u$-(suppression) unconditional.

		%\noindent{\bf Theorem GPT} \label{gpt} (\cite{GPT, Tem})
		%Let $\mathcal B$ a basis of a Banach space $\mathbb X$. The following are equivalent:
		%\begin{itemize}
		%	\item If $\mathcal B$ is $w$-greedy with constant $C_g$, the basis is unconditional with constant $K_u\leq C_g$ and $w$-democratic with constant $C_d\leq C_g$.
		%	\item If $\mathcal B$ is unconditional with constant $K_u$ and $w$-democratic with constant $C_d$, the the basis is $w$-greedy with constant $C_g\leq K_u + K_u^3C_d$.
		%\end{itemize}
		
		Other important $w$-type greedy basis in this context is the $w$-almost-greedy basis.
		
		\begin{defn}
			We say that $\mathcal B$ is $w$-\textbf{almost-greedy} if there exists a constant $C\geq 1$ such that
			\begin{eqnarray}\label{walm}
			\Vert x-\mathcal G_m(x)\Vert \leq C\tilde{\sigma}_{w(A_m(x))}^w,\; \forall x\in\mathbb X, \forall m\in\mathbb N,
			\end{eqnarray}
			where 
			$$\tilde{\sigma}_\delta^w(x,\mathcal B)_{\mathbb X}=\tilde{\sigma}_{\delta}^w (x):= \inf\lbrace \Vert x-P_A(x)\Vert : A\in\mathbb N^{<\infty}, w(A)\leq \delta\rbrace.$$
			We denote by $C_{al}$ the least constant that satisfies \eqref{walm} and we say that $\mathcal B$ is $C_{al}$-$w$-almost-greedy.
		\end{defn}
		
		\begin{rem}
			If $w\equiv 1$, that is, $w(A)=\vert A\vert$, we recover the classical definition of almost-greediness (resp.~greediness, democracy and Property (A)-see the definition below), and we will say that $\mathcal B$ is almost-greedy (resp.~greedy, democratic, has the Property (A)).
		\end{rem}
		In the classical sense, that is, when $w\equiv 1$, S. J. Dilworth, N. J. Kalton, D. Kutzarova and V. N. Temlyakov gave in \cite{DKKT} a characterization of almost-greedy bases in terms of the quasi-greediness and democracy. Recently, S. J. Dilworth, D. Kutzarova, V. N. Temalykov and B. Wallis, in \cite{DKTW}, gave a characterization of $w$-almost-greedy bases in terms of quasi-greedy and $w$-democratic bases. 
		
		It is well known that, even for $w \equiv 1$,
		the $w$-democracy and unconditionality (resp.~quasi-greediness),
		cannot be used to determine whether a given basis is $w$-greedy (resp.~$w$-almost-greedy) with constant 1. For the weight $w \equiv 1$, F. Albiac and P. Wojtaszczyk introduced in \cite{AW}
		the so called Property (A) (defined below) in order to obtain finer estimate for the greedy constant $C_g$
		(and, in particular, to characterize bases with $C_g=1$). 
		% Some years later, in \cite{DKOSS},
		% S.J. Dilworth, D. Kutzarova, E. Odell, T. Schlumprecht and A. Zsak gave a reformulation of
		% this property to obtain a generalization of the theorem of Albiac and Wojtaszczyk and, using
		% this reformulation, Albiac and J.L. Ansorena obtained in \cite{AA2} the constant $C_{al}=1$.
		% The definition of this Property (A) is the following:
		The results of \cite{AW} were further generalized in \cite{DKOSS};
		in \cite{AA2}, the Property (A) was used to estimate the almost-greedy constant $C_{al}$.
		
		% \begin{defn}
		% We say that $\mathcal B$ has the Property (A) if
		% $$\Vert x+t\mathbf{1}_{\varepsilon A}\Vert \leq C\Vert x+t\mathbf{1}_{\eta B}\Vert,$$
		% 
		% for any $x\in\mathbb X$, $\vert A\vert = \vert B\vert<\infty$, $A\cap B=\emptyset$, $\supp(x)\cap (A\cup B)=\emptyset$, $\varepsilon, \eta \in \lbrace \pm 1\rbrace$ and $t=\sup_j \vert e_j^*(x)\vert$.
		% \end{defn}

		% \begin{rem}
		% In this definition, we take $\vert A\vert=\vert B\vert$, but this condition is equivalent to take $\vert A\vert \leq \vert B\vert$: assume that we have the Property (A) with $\vert A\vert = \vert B\vert$. Now, take $\vert A\vert \leq \vert B\vert$, then
		% $$\Vert x+t\mathbf{1}_{\varepsilon A}\Vert \leq C_a\Vert x+t\mathbf{1}_{E}\Vert \leq C_aK_b\Vert x+t\mathbf{1}_{F}\Vert \leq C_a^2K_b\Vert x+t\mathbf{1}_{\eta B}\Vert,$$
		% where $F>\supp(x)\cup A\cup B$, $\vert F\vert = \vert B\vert$ and $E\subset F$ is the set with $\vert E\vert=\vert A\vert$ and has the first $\vert A\vert$ indices of $F$.
		% \end{rem}

		% We are going to introduce the concept of Property (A) with weights to get the constant 1 for the $w$-greedy and $w$-almost-greedy bases for general weights.
		
		Throughout the paper, we will be using a weighted version of Property (A):
		
		\begin{defn}\label{d:w_prop_A}
			We say that $\mathcal B$ satisfies the $w$-\textbf{Property (A)} if there exists a constant $C\geq 1$ such that
			\begin{eqnarray}\label{wA}
			\Vert x+t\mathbf{1}_{\varepsilon A}\Vert \leq C\Vert x+t\mathbf{1}_{\eta B}\Vert,
			\end{eqnarray}
			for any $x\in\mathbb X$, for any $A,B \in \mathbb N^{<\infty}$such that $w(A)\leq w(B)$, $A\cap B=\emptyset$, $\supp(x)\cap (A\cup B)=\emptyset$, for any $\varepsilon, \eta \in \lbrace \pm 1\rbrace$ and $t \geq \sup_j \vert e_j^*(x)\vert$.
			We denote by $C_a$ the least constant that satisfies \eqref{wA} and we say that $\mathcal B$ has the $C_a$-$w$-Property (A).
		\end{defn}
		
		\begin{rem}
			The definition of $w$-Property (A) was motivated by the ``classical'' Property (A)
			(introduced in \cite{AW}), which states that
			$$
			\Vert x+t\mathbf{1}_{\varepsilon A}\Vert \leq C\Vert x+t\mathbf{1}_{\eta B}\Vert,
			$$
			whenever %$x\in\mathbb X$,
			$|A| = |B|<\infty$, $A\cap B=\emptyset$, $\supp(x)\cap (A\cup B)=\emptyset$,
			$\varepsilon, \eta \in \lbrace \pm 1\rbrace$ and $t \geq \sup_j \vert e_j^*(x)\vert$.
			Proposition \ref{w-equivalent} shows that the classical Property (A) is equivalent to the $w$-Property (A)
			if $0 < \inf_n w_n \leq \sup_n w_n < \infty$.
			% If $w_n = 1$ for every $n$, the two properties are ``equivalent up to a constant depending on $K_b$.''
			% Clearly the $w$-Property (A) implies Property (A) with the same constant.
			% Conversely, suppose our basis ${\mathcal{B}}$ has Property (A) with constant $C_a$.
			% Now suppose $\vert A\vert \leq \vert B\vert$. 
			% Pick $F = \{N \, \ldots, N + |B| - 1\} >\supp(x)\cup A\cup B$
			% and let $E = \{N, \ldots, N + |A| - 1\}$. Then
			% $$\Vert x+t\mathbf{1}_{\varepsilon A}\Vert \leq C_a\Vert x+t\mathbf{1}_{E}\Vert \leq C_aK_b\Vert x+t\mathbf{1}_{F}\Vert \leq C_a^2K_b\Vert x+t\mathbf{1}_{\eta B}\Vert .$$
		\end{rem}

		Another way of estimating the efficiency of greedy approximation is to compare the rate of convergence
		with straightforward Schauder approximation. To this end we consider
		\textbf{$w$-partially-greedy bases}. In \cite{DKKT}, the authors defined the partially-greedy bases as those satisfying $$\Vert x-\mathcal{G}_m(x)\Vert \leq C\Vert x-S_m(x)\Vert,\; \forall x\in\mathbb X,\forall m\in\mathbb N,$$
		for some positive and absolute constant $C$.
		% $(S_m)_{m=1}^\infty$ is the algorithm of the partial sums.
		Moreover, they proved that $\mathcal B$ is partially-greedy if and only if $\mathcal B$ is quasi-greedy and conservative (that is, $\Vert \mathbf{1}_A\Vert \lesssim \Vert \mathbf{1}_B\Vert$ for all pair of finite sets $A,B$ such that $A<B$ and $\vert A\vert\leq\vert B\vert$). Here, we present the notion of $w$-partially-greedy bases and we characterize these bases using $w$-conservative bases.

		The paper is structured as follows. In Section \ref{s:characterize} we describe the $w$-greedy and
		$w$-almost-greedy bases in terms of their other properties (such as $w$-Property (A),
		unconditionality, or being quasi-greedy). The main results are Theorems \ref{1wgreedy} and
		\ref{1walmostgreedy}.
		
		In Section \ref{s:prop_A}, we collect basic facts about the $w$-Property (A). 
		% We move on to studying $w$-semi-greedy bases.
		In addition, we consider the $w$-semi-greedy bases -- that is, the bases where the Chebyshev
		greedy approximands are optimal. It turns out (Theorem \ref{t:semi_greedy}) such bases necessarily
		possess the $w$-Property (A).
		
Section \ref{s:properties_C_and_D} is devoted to properties (C) and (D), which arise naturally in
the study of quasi-greedy bases. In particular, it is shown that $w$-superdemocracy and
Property (C) imply $w$-Property (A) (Proposition \ref{propCS}).
However, superdemocracy does not imply Property (C) (Example \ref{ex:no_D}).
Further, we show that any $w$-semi-greedy basis has Property (C) if the weight $w$
is equivalent to a constant (Proposition \ref{p:SG_implies_C}).

		In Section \ref{s:partially_greedy}, we compare the efficiency of greedy approximation with that of the canonical
		basis projections. This gives rise to the notion of an $w$-partially-greedy basis;
		such bases are characterized in Theorem \ref{t:partially_greedy}.

% 		In Section \ref{s:equiv_to_const}, we return to our investigation of the $w$-Property (A)
% 		in the particular case when the weight $w$ is equivalent to constants. We show that,
% 		in this setting, the $w$-Property (A) is equivalent to the ``classical'' Property (A)
% 		(Theorem \ref{propA1}). 

		% The very short Section \ref{s:c_0} shows that, if the weight $w$ belongs to $c_0$,
		% and a basis $\mathcal B$ is $w$-democratic and conservative, then it is
		% equivalent tot he canonical basis of $c_0$. 
		Finally, in Section \ref{s:questions} and Section \ref{appendix}
		we state some open questions related to our results, and prove some basic lemmas used throughout the paper. % the topic at hand.
		
		We freely use the standard ``greedy'' terminology.
		The reader can consult e.g. \cite{Tem} for more information.
		
% 	\end{section}
	
	\section{Characterization of $w$-greedy and $w$-almost-greedy bases}\label{s:characterize}
	
	% In this section, we are going to prove the following theorems.
	
	In this section we describe the $w$-(almost)-greediness of a basis in terms of its $w$-Property (A)
	and unconditionality (resp.~quasi-greediness). The corresponding results for the constant
	weight $w\equiv 1$ can be found, for instance, in \cite{Tem}.
	
	\begin{thm}\label{1wgreedy}
		Let $\mathcal B$ be a basis of a Banach space $\mathbb X$. 
		\begin{itemize}
			\item[a)] If $\mathcal B$ is $C_g$-$w$-greedy, then the basis is $K_u$-unconditional and has the $C_a$-$w$-Property (A) with constants $K_u\leq C_g$ and $C_a\leq C_g$.
			\item[b)] If $\mathcal B$ is $K_u$-unconditional and has the $C_a$-$w$-Property (A), then the basis is $C_g$-$w$-greedy with $C_g\leq K_uC_a$.
		\end{itemize}
	\end{thm}
	
	\begin{thm}\label{1walmostgreedy}
		Let $\mathcal B$ be a basis of a Banach space $\mathbb X$. 
		\begin{itemize}
			\item[a)] If $\mathcal B$ is $C_{al}$-$w$-almost-greedy, then the basis is $C_{q}$-quasi-greedy and has the $C_{a}$-$w$-Property (A) with constants $C_q\leq C_{al}$ and $C_a\leq C_{al}$.
			\item[b)] If $\mathcal B$ is $C_q$-quasi-greedy and has the $C_a$-$w$-Property (A), then the basis is $C_{al}$-$w$-almost-greedy with $C_{al}\leq C_qC_a$.
		\end{itemize}
	\end{thm}
	
Later (Proposition \ref{RW-infty}) we will see examples of bases with $w$-Property (A)
for a certain weight $w$, but failing the ``classical'' Property (A).

	% To prove these theorems, we need a previous lemma based on the truncation operator that we can find in \cite{BBG} in the context of the Lebesgue-type inequalities. 

	% For our purpose we need the following reformulation of the $w$-Property (A)
	% based on the reformulation given by Albiac and Ansorena in \cite{AA2}.
	
	For further use, we need the following reformulation of the $w$-Property (A)
	(inspired by \cite{AA2}).
	
	\begin{prop}\label{p:prop_A_projections}
		A basis $\mathcal B$ has the $C_a$-$w$-Property (A) if and only if 
		\begin{eqnarray}\label{A}
		\Vert x\Vert \leq C_a\Vert x-P_A(x)+\mathbf{1}_{\eta B}\Vert,
		\end{eqnarray}
		for any $x\in\mathbb X$ with $\sup_j |e_j^*(x)| \leq 1$, $A,B\in\mathbb N^{<\infty}$, $w(A)\leq w(B)$, $B\cap \supp(x)=\emptyset$ and $\eta\in\lbrace\pm 1\rbrace$.
	\end{prop}
	
	The proof requires a technical result.
	
	\begin{lem}\label{l:appr_fun_supp}
		% For any $x \in \mathbb X \backslash \{0\}$, any finite set $D \subset \N$, and any $\varepsilon > 0$,
		Suppose $D$ is a finite subset of $\N$, and $x \in \mathbb X \backslash \{0\}$
		satisfies $\supp(x) \cap D = \emptyset$.
		Then for any $\varepsilon > 0$ there exists a finitely supported $y \in \mathbb X$, so that
		$\Vert x - y\Vert < \varepsilon$, $\supp(y) \cap D = \emptyset$,
		and $\max_j |e_j^*(x)| = \max_j |e_j^*(y)|$.
	\end{lem}
	
	\begin{proof}
		It suffices to consider $\varepsilon < 1/(2c_2)$.
		By scaling, we can assume that $\max_j |e_j^*(x)| = 1$ (then $\Vert x \Vert \geq 1/c_2$).
		% The set $C$ define the ``coordinate projection''
		%Consider the ``coordinate projection'' $P$ associated with the set $D$:
		%$P : {\mathbb{X}} \to {\mathbb{X}} : x \mapsto \sum_{j \in D} e_j^*(x) e_j$
		%(we have $\Vert P\Vert \leq |D| c_2^2$). Let $P^\perp = I - P$ be the complementary projection.
		Clearly $P_D(x) = 0$, and $P_D^c(x) = x$. Now set $\delta = \varepsilon/(3c_2^2 \Vert x\Vert)$.
		As ${\textrm{span }}[e_j : j \in \N]$ is dense in $\mathbb X$, there exists a finitely supported $z \in \mathbb X$
		so that $\Vert x - z\Vert < \delta/\Vert P_D^c\Vert$. Let $u = P_D^c(z)$, then
		$\Vert x - u\Vert = \Vert P_D^c(x-z)\Vert < \delta$. For every $j$, $|e_j^*(x-u)| < c_2 \delta$,
		hence $C = \max_j |e_j^*(x)| \in (1 - c_2 \delta, 1 + c_2 \delta)$.
		Now let $y = u/C$. Then $\max_j |e_j^*(y)| = 1$, and
		$$
		\Vert x-y\Vert \leq \Vert x - u\Vert + |1 - C^{-1}| \Vert u\Vert <
		\delta + \frac{c_2 \delta}{1 - c_2 \delta} ( \Vert x\Vert + \delta)< \varepsilon .
		\qedhere
		$$
	\end{proof}

	\begin{proof}[Proof of Proposition \ref{p:prop_A_projections}]
		By Lemma \ref{l:appr_fun_supp}, it suffices to restrict our attention to finitely
		supported vectors $x \in \mathbb X$ only. So, throughout this proof, we assume
		$|\supp(x)| < \infty$.
		
		Suppose that $\mathcal B$ has the $C_a$-$w$-Property (A),
		and $x, A, B, \varepsilon, \eta$ are as in the statement of the proposition with $\sup_j \vert e_j^*(x)\vert \leq 1$.
		% Replacing $A$ by $A \cap \supp(x)$, we can and do assume that $A \subset \supp x$.
		% Let $A\subset \supp(x)$, $B\cap \supp(x)=\emptyset$ such that $w(A)\leq w(B)$ and $\varepsilon,\eta\in\lbrace\pm 1\rbrace$. Then
		Applying the definition of $w$-Property (A) to $P_{A^c} x$, $A$, and $B$, we obtain
		$$
		\Vert P_{A^c}(x)+\mathbf{1}_{\varepsilon A}\Vert \leq C_a\Vert P_{A^c}(x)+\mathbf{1}_{\eta B}\Vert =
		C_a\Vert x - P_A(x)+\mathbf{1}_{\eta B}\Vert .
		$$
		To finish the proof, observe that $x$ belongs to the convex hull of the set
		$\big\{ P_{A^c}(x)+\mathbf{1}_{\varepsilon A} \}_{\varepsilon \in \{\pm 1\}}$.
		% Using now convexity and density, we finish the proof.
		
		Now, suppose \eqref{A}, and prove that the basis $\mathcal B$ has the
		$w$-Property (A) with the same constant. Take $x\in\mathbb X$ % with $|\supp(x)|<\infty$
		and $\sup_j \vert e_j^*(x)| \leq 1$, $A,B\in\mathbb N^{<\infty}$ such that $w(A)\leq w(B)$,
		$A\cap B=\emptyset$, $\supp(x)\cap (A\cup B)=\emptyset$ and $\varepsilon,\eta\in\lbrace \pm 1\rbrace$.
		Define $x' = x+\mathbf{1}_{\varepsilon A}$. Using \eqref{A},
		$$\Vert x+\mathbf{1}_{\varepsilon A}\Vert =
		\Vert x'\Vert\leq C_a\Vert x'-P_A(x')+\mathbf{1}_{\eta B}\Vert = C_a\Vert x+\mathbf{1}_{\eta B}\Vert. \qedhere$$
		% Using density, we have the proof.
	\end{proof}

	\begin{proof}[Proof of Theorem \ref{1wgreedy}:]
		Assume that $\mathcal B$ is $C_g$-$w$-greedy.
		
		\underline{Unconditionality:} Let $x\in\mathbb X$ and $A\subset \supp(x)$. Define $y:= P_{A^c}(x)+\sum_{n\in A}(\alpha + e_n^*(x))e_n$, where $$\alpha > \sup_{j\in A}\vert e_j^*(x)\vert + \sup_{j\in A^c}\vert e_j^*(x)\vert.$$
		
		As $A$ is a greedy set of $y$,
		$$\Vert x-P_{A}(x)\Vert = \Vert y - P_A(y)\Vert \leq C_g \sigma_{w(A)}^w(y)\leq C_g\Vert y-\alpha\mathbf{1}_A\Vert = C_g\Vert x\Vert.$$ Thus, the basis is unconditional with constant $K_u\leq C_g$.
		
		\underline{$w$-Property (A):} Fix $x\in\mathbb X$, take $t\ge \sup_n \vert e_n^*(x)\vert$.
		Consider $\varepsilon,\eta\in\lbrace \pm 1\rbrace$ and finite sets % $A\subset \supp(x)$ and $B$
		$A,B$ such that $A \cap B = \emptyset$,
		$w(A)\leq w(B)$, and $(A \cup B)\cap \supp(x)=\emptyset$. Set $y:= x+t\mathbf{1}_{\varepsilon A} + (t+\delta)\mathbf{1}_{\eta B}$ with $\delta>0$. Hence,
		$$\Vert x+t\mathbf{1}_{\varepsilon A}\Vert = \Vert y-\mathcal{G}_{\vert B\vert}(y)\Vert \leq C_g\sigma_{w(B)}^w(y)\leq C_g\Vert y-t\mathbf{1}_{\varepsilon A}\Vert = C_g\Vert x+(t+\delta)\mathbf{1}_{\eta B}\Vert.$$
		
		Taking $\delta \rightarrow 0$, we obtain that the basis satisfies the $w$-Property (A) with constant $C_a\leq C_g$.\newline
		
		Next we prove that if $\mathcal B$ is $K_u$-unconditional and has the $C_a$-$w$-Property (A), then it is $w$-greedy.
		
		Take $x\in\mathbb X$ and suppose that $A$ is a greedy set of cardinality $m$ for $x\in\mathbb X$ -- that is, $P_A(x)=\mathcal{G}_m(x)$.
		For $\varepsilon > 0$ find $y\in\mathbb X$ such that $\Vert x-y\Vert <\sigma^w_{w(A)}(x)+\varepsilon$,
		with $\supp(y)=B$ and $w(B)\leq w(A)$. Then, taking
		$t:=\min\lbrace \vert e_j^*(x)\vert : j\in A\rbrace$ and $\eta \equiv \sgn\lbrace e_j^*(x)\rbrace$,
		using the the reformulation of the $w$-Property (A) and Lemma \ref{trun}, we obtain that
		\begin{eqnarray*}
			\Vert x-\mathcal G_m(x)\Vert &\leq& C_a\Vert x-P_A(x)-P_{B\setminus A}(x)+t\mathbf{1}_{\eta(A\setminus B)}\Vert = C_a\Vert P_{(A\cup B)^c}(x-y)+t\mathbf{1}_{\eta (A\setminus B)}\Vert \\
			&=&C_a\Vert T_t (I-P_B)(x)\Vert = C_a\Vert T_t (I-P_B)(x-y)\Vert \leq K_uC_a\Vert x-y\Vert.
		\end{eqnarray*}
		Consequently, for any greedy set $A$ we have $\Vert x-P_Ax\Vert \leq K_uC_a\sigma^w_{w(A)}(x)$.
		% with $A$ the greedy set.
		% This completes the proof.
	\end{proof}

	\begin{proof}[Proof of Theorem \ref{1walmostgreedy}]
		Assume that $\mathcal B$ is $C_{al}$-$w$-almost-greedy.
		
		\underline{Quasi-greedy:} Since
		$$\Vert x-\mathcal{G}_m(x)\Vert \leq C_{al}\inf\lbrace \Vert x-\sum_{n\in B}e_n^*(x)e_n\Vert : w(B)\leq w(A_m(x)), B\in\mathbb N^{<\infty}\rbrace,$$
		we can select $B=\emptyset$. Then, we obtain that $\Vert x-\mathcal G_m(x)\Vert \leq C_{al}\Vert x\Vert$,
		hence the basis is quasi-greedy with constant $C_q\leq C_{al}$.
		
		\underline{$w$-Property (A):} We can use the same argument as in Theorem \ref{1wgreedy}.
		
		Now, we will prove that if $\mathcal B$ is $C_q$-quasi-greedy and has the $C_a$-$w$-Property (A), then it is $w$-almost-greedy.
		
		For $x\in\mathbb X$, let $A$ be a greedy set of cardinality $m$. % and $P_A(x)=\mathcal{G}_m(x)$ the greedy sum.
		For $\varepsilon>0$, find $B$ such that $\Vert x-P_B(x)\Vert <\tilde{\sigma}_{w(A)}^w(x)+\varepsilon$, with
		$w(B)\leq w(A)$. Then, taking
		$t:=\min\lbrace \vert e_j^*(x)\vert : j\in A\rbrace$ and
		$\eta\equiv \sgn\lbrace e_j^*(x)\rbrace$, using the reformulation of the $w$-Property (A)
		and Lemma \ref{trun}, 
		\begin{eqnarray*}
			\Vert x-\mathcal G_m(x)\Vert &\leq& C_a\Vert P_{(A\cup B)^c}(x-y)+t\mathbf{1}_{\eta (A\setminus B )}\Vert \\
			&=&C_a\Vert T_t (I-P_B)(x)\Vert \leq C_qC_a\Vert x-P_B(x)\Vert.
		\end{eqnarray*}
		This gives that, for any greedy set $A$, $\Vert x-P_A(x)\Vert \leq C_qC_a\tilde{\sigma}_{w(A)}(x)$ as desired.
	\end{proof}

	\begin{rem}\label{r:def_prop_A}
		In this paper, we focus on the situation when $\mathcal B$ is a Schauder basis.
		However, the $w$-Property (A) can be defined for any complete biorthogonal system
		satisfying the conditions (A1)-(A3); the proof of Proposition \ref{p:prop_A_projections} goes through as well.
		
		Moreover, in the definition of the $w$-Property (A), it suffices to show that
		\eqref{wA} holds for with $\max_j |e_j^*(x)| = t$. More specifically,
		the following four statements are equivalent:
		\begin{itemize}
			\item[(a)] $\mathcal B$ satisfies the $w$-Property (A) (see Definition \ref{d:w_prop_A}).
			\item[(b)] There exists a constant $C$ so that $\Vert x\Vert \leq C\Vert x-P_A(x)+t{\mathbf{1}}_{\eta B}\Vert$
			for any $\eta\in\lbrace\pm 1\rbrace$, $x\in\mathbb X$, $B\cap \supp(x)=\emptyset$, $w(A)\leq w(B)$
			and $t\geq\sup_j\vert e_j^*(x)\vert$.
			\item[(c)] There exists a constant $C^\prime$ so that
			$\Vert x+s\mathbf{1}_{\varepsilon A}\Vert \leq C\Vert x+s\mathbf{1}_{\eta B}\Vert$
			for any $x\in\mathbb X$, $w(A)\leq w(B)$, $A\cap B=\emptyset$, $\supp(x)\cap (A\cup B)=\emptyset$,
			$\varepsilon, \eta \in \lbrace \pm 1\rbrace$ and $s = \sup_j \vert e_j^*(x)\vert$.
			\item[(d)] There exists a constant $C^\prime$ so that $\Vert x\Vert \leq
			C^\prime \Vert x-P_A(x)+s{\mathbf{1}}_{\eta B}\Vert$ for any $x\in\mathbb X$, $\eta\in\lbrace\pm 1\rbrace$,
			$B\cap \supp(x)=\emptyset$, $w(A)\leq w(B)$ and $s = \sup_j\vert e_j^*(x)\vert$.
		\end{itemize}
		Indeed, the implications (a) $\Rightarrow$ (c) and (b) $\Rightarrow$ (d) (with $C^\prime = C$)
		are immediate. The equivalence (a) $\Leftrightarrow$ (b) (with the same constant $C$)
		has been established in Proposition \ref{p:prop_A_projections}.
		Minor adjustments to that argument give us (c) $\Leftrightarrow$ (d).
		
		To establish (d) $\Rightarrow$ (b), take $x,A,B,\eta$ as in (b) and $t\geq \sup_j\vert e_j^*(x)\vert$.
		As before, we can assume that $x$ is finitely supported.
		Find $k$ so that $\vert e_k^*(x) \vert = \sup_j\vert e_j^*(x)\vert$.
		By replacing $x$ by $-x$ if necessary, we can assume $s = e_k^*(x) \geq 0$.
		Let $c = t-s$, and consider
		$$
		x^\prime = x + c e_k = \sum_{j \in \supp(x) \backslash \{k\}} e_j^*(x) e_j + t e_k .
		$$
		Note that $\Vert x - x^\prime\Vert \leq c c_2 \leq t c_2$.
		Furthermore, $x^\prime - P_A (x^\prime)$ equals either $x - P_A(x)$ (if $k \in A$), or
		$x - P_A(x) + c e_k$ (if $k \notin A$). In either case,

		$$
		\Vert x - P_A(x) +t{\mathbf{1}}_{\eta B}\Vert \geq 
		\Vert x^\prime-P_A(x^\prime)+t{\mathbf{1}}_{\eta B}\Vert - t c_2 .
		$$
		
		By (d), we have $\Vert x^\prime \Vert \leq C^\prime \Vert x^\prime-P_A(x^\prime)+t{\mathbf{1}}_{\eta B}\Vert$.
		By the above,
		$$
		\Vert x \Vert - t c_2 \leq C^\prime (\Vert x - P_A(x) + t{\mathbf{1}}_{\eta B}\Vert + t c_2)
		$$
		As $\Vert x - P_A(x) + t{\mathbf{1}}_{\eta B}\Vert \geq t c_2^{-1}$, we conclude that
		$\Vert x \Vert \leq (C^\prime + 2 c_2^2) \Vert x - P_A(x) + t{\mathbf{1}}_{\eta B}\Vert$.
	\end{rem}

	\section{Some remarks on the $w$-Property (A)}\label{s:prop_A}
	% We star this section with the following triviality.

	\begin{defn}\label{d:equivalent}
		Let $v=(v_n)_{n=1}^\infty$ and $w=(w_n)_{n=1}^\infty$ be weights. We say that $v$ is {\bf equivalent} to $w$, written $v\approx w$, whenever there exist positive real constants $0<a\leq b<\infty$ satisfying
		$$av_n\leq w_n\leq bv_n\;\;\;\text{ for all }n\in\mathbb{N}.$$\end{defn}
	
	%\begin{definition}A weight $w=(w_n)_{n=1}^\infty\in(0,\infty)^\mathbb{N}$ is said to be {\bf nonincreasing} whenever $w_{n+1}\leq w_n$ for all $n\in\mathbb{N}$. It is called {\bf essentially decreasing} whenever it is equivalent to a nonincreasing weight.\end{definition}
	
	\begin{prop}\label{w-equivalent}
		Let $v, w$ weights and suppose that $v\approx w$.
		% $v\approx w$ with constants $0<a\leq b<\infty$ as described above.
		Then every basis with the $w$-Property (A) also has the $v$-Property (A).\end{prop}
	
	\begin{proof}
		Let $x\in\mathbb X$ with $\vert\supp(x)\vert<\infty$ and $\sup_j\vert e_j^*(x)\vert \leq 1$, $A$ and $B$ finite satisfying $v(A)\leq v(B)$, $A\cap B=\emptyset$, $\supp(x)\cap (A\cup B)=\emptyset$ and $\varepsilon,\eta\in\lbrace\pm 1\rbrace$. We set
		$$\Gamma=\left\{n\in A:w_n\geq w(B)\right\}.$$
		Observe that
		$$w(A)\leq b\cdot v(A)\leq b\cdot v(B)\leq\frac{b}{a}\cdot w(B),$$
		which gives us
		$$w(A)\geq w(\Gamma)\geq\vert\Gamma\vert\cdot w(B)\geq \vert\Gamma\vert\cdot\frac{a}{b}\cdot w(A),$$
		and hence $\vert\Gamma\vert\leq b/a$. Next, we give the following partition of $A\setminus\Gamma$: $A_1<\ldots<A_m,$
		so that for each $i=1,\ldots,m$, the set $A_i$ is a maximal such that $w(A_i)\leq w(B)$. Due to maximality,
		$$w(B)<w(A_i)+w(A_{i+1})\text{ for all }i=1,\ldots,m-1.$$
		Thus,
		$$(m-1)\cdot w(B)<\sum_{i=1}^{m-1}\left[w(A_i)+w(A_{i+1})\right]<2\cdot w(A\setminus\Gamma)\leq 2\cdot w(A)\leq\frac{2b}{a}\cdot w(B).$$
		This gives us
		$$m\leq\frac{2b}{a}+1.$$
		
		Hence, using the bounds of $\vert \Gamma\vert$, $m$ and the condition of the $w$-Property (A),
		\begin{eqnarray*}
			\Vert x+\mathbf{1}_{\varepsilon A}\Vert &\leq& \Vert \mathbf{1}_\Gamma\Vert + \Vert x+\sum_{i=1}^{m}\mathbf{1}_{\varepsilon A_i}\Vert \leq \sum_{n\in\Gamma}\Vert e_n\Vert + \sum_{i=1}^m\Vert \frac{x}{m}+\mathbf{1}_{\varepsilon A_i}\Vert\\
			&\leq& c_2^2\vert\Gamma\vert\Vert x+\mathbf{1}_{\eta B}\Vert + C_a m \Vert \frac{x}{m}+\mathbf{1}_{\eta B}\Vert \leq \frac{c_2^2 b}{a}\Vert x+\mathbf{1}_{\eta B}\Vert + C_a\Vert x+m\mathbf{1}_{\eta B}\Vert\\
			&\leq& \frac{c_2^2 b}{a}\Vert x+\mathbf{1}_{\eta B}\Vert + C_a m \Vert x+\mathbf{1}_{\eta B}\Vert + C_a(m-1)\Vert x\Vert\\
			&\leq& \frac{c_2^2 b}{a}\Vert x+\mathbf{1}_{\eta B}\Vert + C_a m \Vert x+\mathbf{1}_{\eta B}\Vert + C_a^2(m-1)\Vert x+\mathbf{1}_{\eta B}\Vert\\
			&\leq & \left(\frac{c_2^2 b+2bC_a^2}{a}\right)\Vert x+\mathbf{1}_{\eta B}\Vert.
		\end{eqnarray*}
	\end{proof}

	\begin{rem}\label{cons-equiv}
		In a similar fashion, one can show that, if the weights $w$ and $v$ are equivalent,
		then any $w$-democratic ($w$-superdemocratic, $w$-conservative --
		for the definitions, see below) basis is also $v$-democratic
		(resp.~$v$-superdemocratic or $v$-conservative).
		% The same argument works to show that $w$-superdemocracy (respect. $w$-democracy and $w$-conservative) is $v$-superdemocratic (respect. $v$-democracy and $v$-conservative) whenever $v\approx w$.
	\end{rem}
	
	\begin{rem}The converse to Proposition \ref{w-equivalent} does not hold in general.
		For example, suppose the weights $w, v$ belong to $\ell_1$.
		By \cite{DKTW}, the family of $w$-democratic (or $v$-democratic) bases consists precisely
		of those bases which are equivalent to the canonical basis of $c_0$.
		However, $w$ and $v$ need not be equivalent.
		% not all weights in $\ell_1$ are equivalent, even though every seminormalized $w$-democratic basis, $w\in\ell_1$, is equivalent to the unit vector basis of $c_0$.
	\end{rem}

	The rest of this section is motivated by the recent definition of $w$-semi-greedy bases introduced in \cite{DKTW}.
	To give this notion, we need the Chebyshev Greedy Algorithm:
	for $x\in\mathbb X$ and $m \in \N$, if $A_m(x)$ is the greedy set of $x$ with cardinality $m$, we define the Chebyshev Greedy Approximand of order $m$ as any $\overline{G}_m(x)\in span\lbrace e_i : i\in A_m(x)\rbrace$ such that
	$$\Vert x-\overline{G}_m(x)\Vert = \min\lbrace \Vert x-\sum_{n\in A_m(x)}b_n e_n\Vert : b_n\in\mathbb R\rbrace.$$
	% $(\overline{G}_m)_{m=1}^\infty$, like
	% $$\Vert x-\overline{G}_m(x)\Vert = \min\lbrace \Vert x-\sum_{n\in A_m(x)}a_n e_n\Vert : a_n\in\mathbb R\rbrace,$$
	% with $A_m(x)=\supp(\mathcal G_m(x))$.
	\begin{defn}
		We say that $\mathcal B$ is \textbf{$w$-semi-greedy} if there exists a constant $C\geq 1$ such that
		\begin{eqnarray}\label{csemi}
		\Vert x-\overline G_m(x)\Vert \leq C\sigma_{w(A_m(x))}^w(x),\; \forall x\in\mathbb X, \forall m\in \mathbb N.
		\end{eqnarray}
		We denote by $C_{sg}$ the least constant that satisfies \eqref{csemi} and we say that $\mathcal B$ is $C_{sg}$-$w$-semi-greedy.
	\end{defn}
	
	By \cite{DKTW}, any $w$-semi-greedy basis is $w$-superdemocratic.
	
	\begin{defn}
		We say that $\mathcal B$ is $w$-\textbf{superdemocratic} if there exists a constant $C\geq 1$ such that
		\begin{eqnarray}\label{supe}
		\Vert \mathbf{1}_{\varepsilon A}\Vert \leq C\Vert \mathbf{1}_{\eta B}\Vert,
		\end{eqnarray}
		for any $A,B\in\mathbb N^{<\infty}$ with
		$w(A)\leq w(B)$ and $\varepsilon, \eta \in \lbrace \pm 1\rbrace$.
		
		We denote by $C_s$ the least constant that satisfies \eqref{supe} and we say that $\mathcal B$ is $C_s$-superdemocratic.
	\end{defn}
	
	Here, we show that 
	\begin{eqnarray}\label{impli}
	w-\text{semi-greedy}\Rightarrow w-\text{Property (A)} \Rightarrow w-\text{superdemocracy}.
	\end{eqnarray}

	\begin{thm}\label{t:semi_greedy}
		If a basis $\mathcal{B}$ is $w$-semi-greedy, then $\mathcal B$ has the $w$-Property (A).
	\end{thm}
	
	\begin{proof}
		Assume that $\Vert x-\overline{G}_m(x)\Vert \leq C_{sg}\sigma_{w(A_m(x))}^w$ for any $x\in\mathbb N$ and $m\in\mathbb N$.
		
		We take $\varepsilon, \eta, A, B$ and $x$ in the conditions of the definition of the $w$-Property (A). In all of the following cases we consider $x\in\mathbb X$ such that $\vert\supp(x)\vert<\infty$ and $\sup_n\vert e_n^*(x)\vert\leq 1$. \newline
		
		\underline{\textbf{Case 1:}} $\sum_{n=1}^\infty w_n = \infty$ and $\sup_n w_n <\infty$.
		
		\textbf{Case 1.1:} $w(B)>\lim\sup_{n\rightarrow\infty} w_n$. Since $\sum_{n}w_n = \infty$, we can choose $E$ and $n_0\in\mathbb N$ with $\min E> \max (A\cup B\cup \supp(x))$ and $n_0>\max E$ such that $$w (E)\leq w(B)<w(E)+w_{n_0}<2w(B).$$
		Set $F:=E\cup\lbrace n_0\rbrace$. Then, $w (E)\leq w(B)<w (F)<2w(B)$.\newline
		
		We define the element $z:= x+\mathbf{1}_{\varepsilon A}+(1+\delta)\mathbf{1}_F$.
		For any scalar sequence $(f_n)_{n \in F}$, we have
		$\Vert x+\mathbf{1}_{\varepsilon A}\Vert \leq K_b\Vert x+\mathbf{1}_{\varepsilon A}+\sum_{n\in F}f_n e_n\Vert$.
		As the basis $\mathcal B$ is $w$-semi-greedy with constant $C_{sg}$, and $w(A)\leq w(B)<w(F)$,
		we conclude that
		$$\inf_{f_n} \Vert x+\mathbf{1}_{\varepsilon A}+\sum_{n\in F}f_n e_n\Vert
		\leq C_{sg}\sigma_{w (F)}^w(z)\leq C_{sg}\Vert x+(1+\delta)\mathbf{1}_F\Vert.$$
		Consequently, $\Vert x+\mathbf{1}_{\varepsilon A}\Vert \leq K_bC_{sg}\Vert x+(1+\delta)\mathbf{1}_F\Vert$.
		Taking $\delta\rightarrow 0$,
		\begin{eqnarray}\label{1}
		\Vert x+\mathbf{1}_{\varepsilon A}\Vert &\leq& K_b C_{sg}\Vert x+\mathbf{1}_F\Vert \leq K_bC_{sg}\Vert x+\mathbf{1}_E\Vert + K_bC_{sg}\Vert e_{n_0}\Vert\\\nonumber
		&\leq& K_bC_{sg}\Vert x+\mathbf{1}_E\Vert + K_bC_{sg} c_2 \leq K_bC_{sg}(\Vert x+\mathbf{1}_{\eta B}\Vert +\Vert \mathbf{1}_{\eta B}\Vert + \Vert \mathbf{1}_E\Vert)+K_bC_{sg} c_2.
		\end{eqnarray}
		
		Now, we set $y:= \mathbf{1}_{\eta B}+(1+\delta)\mathbf{1}_F$. Reasoning as before, we obtain
		$$\Vert \mathbf{1}_{\eta B}\Vert \leq K_b \inf_{c_n} \Vert \mathbf{1}_{\eta B}+\sum_{n\in F}c_ne_n\Vert
		\leq K_bC_{sg}\sigma_{w(F)}^w(y)\leq K_bC_{sg}\Vert (1+\delta)\mathbf{1}_F\Vert.$$
		Sending $\delta\rightarrow 0$, we obtain
		\begin{eqnarray}\label{2}
		\Vert \mathbf{1}_{\eta B}\Vert \leq K_bC_{sg}\Vert \mathbf{1}_F\Vert \leq K_bC_{sg}\Vert \mathbf{1}_E\Vert + K_bC_{sg}c_2.
		\end{eqnarray}
		
		On the other hand, taking $s:=x+(1+\delta)\mathbf{1}_{\eta B} +\mathbf{1}_E$,
		$$\Vert \mathbf{1}_E\Vert \leq (K_b+1)\Vert x+\sum_{n\in B}b_ne_n + \mathbf{1}_E\Vert \leq C_{sg}(K_b+1)\sigma_{w(B)}^w(s)\leq C_{sg}(K_b+1)\Vert x+(1+\delta)\mathbf{1}_{\eta B}\Vert.$$
		
		Then, taking $\delta\rightarrow 0$,
		\begin{eqnarray}\label{3}
		\Vert \mathbf{1}_E\Vert\leq C_{sg}(K_b+1)\Vert x+\mathbf{1}_{\eta B}\Vert.
		\end{eqnarray}
		
		Finally, using \eqref{1}, \eqref{2} and \eqref{3}, the basis satisfies the $w$-Property (A) with constant
 $K=O(C_{sg}^3 K_b^3 c_2)$.\newline
		
		\textbf{Case 1.2:} $w(A)\leq w(B)\leq \lim\sup_{n\rightarrow\infty}w_n$. Using Proposition 3.5 of \cite{DKTW}, $$\max\lbrace \Vert \mathbf{1}_{\varepsilon A}\Vert, \Vert \mathbf{1}_{\eta B}\Vert\rbrace\leq 2K_bC_{sg}c_2.$$ 
		Since $1=\vert e_j^*(x+\mathbf{1}_{\eta B})\vert \leq \Vert e_j^*\Vert \Vert x+\mathbf{1}_{\eta B}\Vert\leq c_2 \Vert x+\mathbf{1}_{\eta B}\Vert$ for $j\in B$, then
		\begin{equation} \begin{split}
		\Vert x+\mathbf{1}_{\varepsilon A}\Vert &\leq \Vert x+\mathbf{1}_{\eta B}\Vert + \Vert \mathbf{1}_{\eta B}\Vert + \Vert \mathbf{1}_{\varepsilon A}\Vert \\ &\leq \Vert x+\mathbf{1}_{\eta B}\Vert + 4 K_bC_{sg}c_2 \leq (4 K_bC_{sg}c_2^2+1)\Vert x+\mathbf{1}_{\eta B}\Vert.
		\end{split} \end{equation}
		
		\underline{\textbf{Case 2:}} If $\sum_n w_n<\infty$ or $\sup_n w_n =\infty$, using the Proposition 3.5 of \cite{DKTW}, $\mathcal B$ is equivalent to the canonical basis of $c_0$ and the result is trivial.
		
	\end{proof}

	\begin{prop}
		If $\mathcal B$ has the $C_a$-$w$-Property (A),
		then $\mathcal B$ is $2C_a$-$w$-superdemocratic. % with $C_s\leq 2C_a$.
	\end{prop}

%{\color{red}{T.O.: We can comment out the following paragraph to save space}.}
%
%Note that % the inequality $C_s \leq \gamma C_a$ fails if $\gamma < 2$.
%$\mathcal B$ need not be $\gamma C_a$-$w$-superdemocratic if $\gamma < 2$.
%For instance, let $\mathbb{X}$ be the completion of $c_{00}$ with respect to the norm
%$\| (x_i) \| = \phi(x_1, x_2) + \sum_{i \geq 3} |x_i|$, where
%$\phi(x_1, x_2) = \max\{ |x_1|, |x_2|, |x_1 + x_2| \}$.
%Let $\mathcal{B}$ be the canonical basis of this space, and set
%the weights $w_1 = w_2 = 2/3$, $w_n = 1$ for $n \geq 3$. One can check that
%$\mathcal B$ has the $w$-Property (A) with constant $1$, but the superdemocracy
%constant equals $2$.

	\begin{proof}
		Take $A,B\in\mathbb N^{<\infty}$ with $w(A)\leq w(B)$, and show that,
for any choice of signs,
$\Vert \mathbf{1}_{\eta A}\Vert \leq 2C_a \Vert \mathbf{1}_{\varepsilon B}\Vert$.
As in \cite[Subsection 4.4]{BBG}, it is enough to prove our inequality for
$\varepsilon \equiv 1$ (otherwise, replace $\mathcal B = \{e_n : n \in \mathbb{N}\}$
by $\lbrace \varepsilon_n e_n : n \in \mathbb{N}\rbrace$).
Since $\mathbf{1}_{\eta A}\in 2S$, where $S=\lbrace \sum_{A'\subset A}\theta_{A'}\mathbf{1}_{A'} : \sum_{A'\subset A}\vert \theta_{A'}\vert \leq 1\rbrace$ (see \cite[Lemma 6.4]{DKO}), it suffices to show that 
		$$\Vert \mathbf{1}_{A'}\Vert \leq C_a \Vert \mathbf{1}_B\Vert, \; \forall A'\subset A.$$
		
		Take $A'\subset A$. Obviously, $\mathbf{1}_{A'} = \mathbf{1}_{A'\setminus B} + \mathbf{1}_{B\cap A'}$. Then, using the $w$-Property (A),
		$$\Vert \mathbf{1}_{A'}\Vert = \Vert \mathbf{1}_{A'\setminus B} + \mathbf{1}_{B\cap A'}\Vert \leq C_a \Vert \mathbf{1}_{B\setminus A'} + \mathbf{1}_{B\cap A'}\Vert = C_a\Vert \mathbf{1}_B\Vert.$$
		We can apply the $w$-Property (A) because
		$$w(A')=w(A'\setminus B)+w(A'\cap B)\leq w(A)\leq w (B) =w(B\setminus A')+w(B\cap A')\Rightarrow w(A'\setminus B)\leq w(B\setminus A').$$
		This completes the proof.
	\end{proof}

	With these results, we have proved the implications \eqref{impli}. 
	
	\begin{rem}
		If $w\equiv 1$, we recover the classical definition of semi-greediness (resp.~superdemocracy), and we will say that $\mathcal B$ is semi-greedy (resp.~superdemocratic).
	\end{rem}

	Improving \cite[Proposition 4.5]{DKTW}, we prove that, in certain cases,
any $w$-superdemocratic basis has to contain a subsequence equivalent to
the canonical basis of $c_0$ ($c_0$-basis henceforth), or even to be equivalent to
such basis.

% 	the $w$-Property (A) of a basis implies that the basis is equivalent to the canonical basis of $c_0$.
	
	\begin{prop}\label{p:find c0}
Suppose a basis $\mathcal{B} = (e_n)_{n=1}^\infty$ is $C_s$-$w$-superdemocratic. % has the $C_a$-$w$-Property (A).
% Suppose that $A\in\mathbb{N}^{<\infty} $, then,
		\begin{enumerate}
			\item[i)] If $A\in\mathbb{N}^{<\infty} $ and $w(A)\leq \lim\sup_{n\rightarrow\infty} w_n$, then
			$\max_{\varepsilon\in\lbrace\pm 1\rbrace}\Vert \mathbf{1}_{\varepsilon A}\Vert \leq c_2C_s$.
			\item[ii)] If $\sup_n w_n = \infty$, then $\mathcal B$ is equivalent to the $c_0$-basis.
			\item[iii)] If $\inf_n w_n = 0$, then there exist $i_1 < i_2 < \ldots$
so that the sequence $(e_{i_k})_{k \in \N}$ is equivalent to the $c_0$-basis.
Moreover, if $\lim_n w_n = 0$, then for any infinite set $A \subset \N$ we can select
$i_1, i_2, \ldots \in A$ with the properties described above.
			\item[iv)] If $\sum_n w_n <\infty$, then $\mathcal B$ is equivalent to the $c_0$-basis.
		\end{enumerate}
	\end{prop}
	
	\begin{proof}
% 		\begin{enumerate}
% 			\item[i)] 
i) % If $A = \{n_1\}$, then clearly $\|\varepsilon e_{n_1}\| \leq c_2$.
Find $n \in \N \backslash A$ so that $w_n > w(A)$, then
$\|{\mathbf 1}_{\varepsilon A}\| \leq C_s \|1_{\{n\}}\| \leq c_2 C_s$.

% Clearly it suffices to consider $|A| \ge 2$.
% 			Choose $\varepsilon\in\lbrace \pm 1\rbrace$. Let $n_1 \in A$ and set $B = A \setminus \{n_1\}$. Then 
% 			$w_{n_1} < w(A) \le \lim\sup_{n\rightarrow\infty} w_n$. Hence we may select $n_2 > A$ such that $w_{n_1} < w(A)\leq w_{n_2}$.
% 			Then,
% 			$$\Vert \mathbf{1}_{\varepsilon A}\Vert \leq C_a\Vert \mathbf{1}_{\varepsilon A} -P_B(\mathbf{1}_{\varepsilon A})+e_{n_0}\Vert =C_a\Vert \pm e_{n_1}+e_{n_2}\Vert \leq 2 c_2C_a.$$
% 			\item[ii)] 

% ii) Choose $n\in\mathbb N$ such that $\sum_{j=n+1}^\infty w_j < w_1$. Suppose that $\min A\geq n+1$. Then, using the same argument that in i),
% 			$$\Vert \mathbf{1}_{\varepsilon A}\Vert \leq C_a \Vert \pm e_{n_1}+e_1\Vert \leq 2c_2C_a,$$
% 			which implies that $\mathcal B$ is equivalent to the canonical basis of $c_0$.
% 			\item[iii)] 

ii) By (i), $\Vert \mathbf{1}_{\varepsilon A}\Vert \leq c_2C_s$ for all choices of signs, which yields the desired equivalence.

iii) Suppose $\inf_n w_n = 0$, and find $i_1 < i_2 < \ldots$ so that $\sum_k w_{i_k} < \infty$.
By convexity, it suffices the existence of a constant $K$ with the property that
the inequality $\|{\mathbf 1}_{\varepsilon E}\| \leq K$ holds
for any finite set $E \subset \{i_1, i_2, \ldots\}$.
To this end, find $N \in \N$ so that $\sum_{j=1}^N w_j \geq \sum_k w_{i_k} < \infty$.
Let $B = \{1, \ldots, N\} \backslash \{i_1, i_2, \ldots\}$,
$D = \{1, \ldots, N\} \cap \{i_1, i_2, \ldots\}$, and
$A = E \backslash D$. Note that $|B|, |D| \leq N$, hence, for every
$\varepsilon \in \{-1,1\}^\N$, $\Vert \mathbf{1}_{\varepsilon B}\Vert ,
 \Vert \mathbf{1}_{\varepsilon D}\Vert\leq c_2 N$.
Then $w(A) \leq w(B)$ and hence
$\Vert \mathbf{1}_{\varepsilon A}\Vert \leq C_s \Vert \mathbf{1}_{\varepsilon B}\Vert \leq c_2 N C_s$.
By the triangle inequality,
$$
\Vert \mathbf{1}_{\varepsilon E}\Vert \leq \Vert \mathbf{1}_{\varepsilon A}\Vert +
\Vert \mathbf{1}_{\varepsilon D}\Vert \leq c_2 N (C_s+1) .
$$
If $\lim_n w_n = 0$, then every infinite $A \subset \N$ contains $i_1 < i_2 < \ldots$
with $\sum_k w_{i_k} < \infty$. It remains to invoke the preceding result.

iv) The proof proceeds as in (iii).
% 		\end{enumerate}
	\end{proof}
	
	From this we immediately obtain:

	\begin{cor}
		If the weight $w$ is unbounded, then a basis has the $w$-Property (A) if and only if it is equivalent
		to the canonical basis of $c_0$.
		% If $w$ is unbounded then every quasi-greedy basis with the $w$-Property (A) for a real Banach space is equivalent to the unit vector basis of $c_0$. In the complex case, the same is true for every unconditional basis with the $w$-Property (A).
	\end{cor}
	
% 	\begin{proof}
% 		Clearly the canonical basis of $c_0$ has $w$-Property (A) for any weight.
% 		Now suppose $w$ is unbounded, and the basis $(e_i)$ has the $w$-Property (A).
% 		By an extreme point argument, it suffices to show the existence of a constant $C > 0$ so that,
% 		for any finite set $A$ and any $\varepsilon \in \{ \pm 1 \}$, we have $\Vert {\mathbf{1}}_{\varepsilon A} \Vert \leq C$.
% 		Take a general set $A$ and find $N$ s.t. $w_N > w(A)$. Due to the $w$-Property (A),
% 		$\Vert {\mathbf{1}}_{\varepsilon A} \Vert \leq C_a \Vert e_N\Vert \leq C_a c$, where $c = \sup_N \Vert e_N\Vert < \infty$.
% 	\end{proof}

	\section{Properties (C) and (D)}\label{s:properties_C_and_D}
	
	 Properties (C) and (D) (discussed below) naturally arise in the study of quasi-greedy bases.

	\begin{defn}\label{d:prop_C}
		We say that $\mathcal B$ satisfies the \textbf{Property (C)} if for any $x\in\mathbb X$, there exists a positive constant $C$ such that
		\begin{eqnarray}\label{cC}
		\min_{j\in \Lambda}\vert e_j^*(x)\vert \Vert \mathbf{1}_{\varepsilon \Lambda}\Vert \leq C\Vert x\Vert,
		\end{eqnarray}
		for any greedy set $\Lambda$ of $x$ and $\varepsilon\in\lbrace\pm 1\rbrace$.
		We denote by $C_u$ the least constant that satisfies \eqref{cC} and we say that $\mathcal B$ has the Property (C) with constant $C_u$.
	\end{defn}
	
	% We know that if $\mathcal B$ is quasi-greedy, then the basis satisfies the Property (C) (see \cite[Lemma 2.3]{BBG}).
	% Moreover, until now, for $w\equiv 1$, we only know that if $\mathcal B$ is superdemocratic and quasi-greedy,
	% then the basis satisfies the Property (A) (see \cite[Lemma 2.2]{BBG} and Appendix).
	% Here, we present the following result.
	
	It is well known any quasy-greedy basis has Property (C) (see \cite[Lemma 2.3]{BBG}).
	Generalizing \cite[Lemma 2.2]{BBG}, we prove that any $w$-superdemocratic basis with
	the Property (C) has the $w$-Property (A).
	
	\begin{prop}\label{propCS}
		If $\mathcal B$ is $C_s$-$w$-superdemocratic and satisfies the Property (C) with constant $C_u$,
		then $\mathcal B$ has the $C_a$-$w$-Property (A) with $C_a\leq 3C_uC_s$.
	\end{prop}
	
	\begin{proof}
		Take $x, A, B, \varepsilon, \eta$ as in the definition of the $w$-Property (A) and assume that $\sup_j\vert e_j^*(x)\vert \leq 1$. Then,
		\begin{eqnarray}\label{n1}
		\Vert x+\mathbf{1}_{\varepsilon A}\Vert \leq \Vert x+\mathbf{1}_{\eta B}\Vert + \Vert \mathbf{1}_{\eta B}\Vert + \Vert \mathbf{1}_{\varepsilon A}\Vert.
		\end{eqnarray}
		
		Using the $w$-superdemocracy and $w(A)\leq w(B)$, we obtain that $\Vert \mathbf{1}_{\varepsilon A}\Vert \leq C_s\Vert \mathbf{1}_{\eta B}\Vert$. Now, we only have to estimate $\Vert \mathbf{1}_{\eta B}\Vert$. For that, we consider the element $y:=x+\mathbf{1}_{\eta B}$. It's clear that $\mathbf{1}_{\eta B}$ is a greedy sum for $y$, so
		\begin{eqnarray}\label{n2}
		\min_{j\in B}\vert e_j^*(y)\vert \Vert \mathbf{1}_{\eta B}\Vert = \Vert \mathbf{1}_{\eta B}\Vert \leq C_u\Vert y\Vert = C_u\Vert x+\mathbf{1}_{\eta B}\Vert.
		\end{eqnarray}
		
		Then, using \eqref{n1} and \eqref{n2},
		$$\Vert x+\mathbf{1}_{\varepsilon A}\Vert \leq \Vert x+\mathbf{1}_{\eta B}\Vert + 2C_sC_u\Vert x+\mathbf{1}_{\eta B}\Vert \leq 3C_sC_u\Vert x+\mathbf{1}_{\eta B}\Vert.$$
		Hence, the basis has the $w$-Property (A) with constant $C_a\leq 3C_uC_s$.
	\end{proof}
	
	\begin{ex}
		We next revisit a ``pathological'' basis constructed in Section 5.5 of \cite{BBG}
		(using some ideas from \cite[Example 4.8]{DKK}): a basis which has the Property (A),
		but fails to be quasi-greedy. The initial proof of the Property (A) was unwieldy.
		Here we present a streamlined proof that the basis has the Property (C), and then
		invoke Proposition \ref{propCS}.
		
		% In section 5.5 of \cite{BBG} an example is presented of a superdemocratic basis in a Banach space that is not quasi-greedy
		% based on \cite[Example 4.8]{DKK} for $w\equiv 1$. It is proved in \cite{BBG} that this basis satisfies the Property (A). Here we present a more elementary proof showing that the basis satisfies the Property (C).	
		
		% We recall the definition of the space and the norm of this example: 
		First recall the construction: $\mathcal D_k$ denote the set of all dyadic intervals $I\subset [0,1]$
		with length $\vert I\vert = 2^{-k}$, and consider $\mathcal D = \cup_{k\geq 0} \mathcal D_k$. Now, we consider the space $\mathfrak f_1^q$ of all real sequences $\textbf a = (a_I)_{I\in\mathcal D}$ such that 
		$$\Vert \textbf a\Vert_{\mathfrak f_1^q} = \left\Vert \left(\sum_I \vert a_I \chi_I^{(1)}\vert^q\right)^{1/q}\right\Vert_{L^1}<\infty,$$
		where $\chi_I^{(1)} = \vert I\vert^{-1}\chi_I$. 
		By \cite{GHO}, the canonical basis $\lbrace e_I\rbrace_{I\in\mathcal D}$ is unconditional and democratic.
		
		For every $N\geq 1$, we shall pick a subset $\lbrace k_1,...,k_N\rbrace\subset \mathbb N_0$ and look at the finite dimensional space $F_N$ consisting of sequences supported in $\cup_{j=1}^N \mathcal D_{k_j}$. We order the canonical basis by $\cup_{j=1}^N \lbrace e_I\rbrace_{I\in\mathcal D_{k_j}}$, so we may as well write their elements as $\textbf a = (a_j)_{j=1}^{\text{dim}\, F_N}$. We also consider in $F_N$ the James norm
		$$\Vert (a_j)\Vert_{J_q}=\sup_{m_0 = 0<m_1<\ldots}\left(\sum_{k\geq 0}\left\vert \sum_{m_k<j\leq m_{k+1}}a_j\right\vert^q\right)^{1/q}.$$
		Now, set in $F_N$ a new norm
		$$\Vert \textbf a\Vert =\max\lbrace \Vert \textbf a\Vert_{\mathfrak f_1^q}, \Vert \textbf a\Vert_{J_q}\rbrace.$$
		Finally, we consider the Banach space $\mathbb X = \oplus_{\ell^1} F_N$ with $\mathcal B$ the consecutive union of the natural bases in $F_N$. 
		
		It's possible to show that $\mathcal B$ is superdemocratic and $\Vert \mathbf{1}_{\varepsilon A}\Vert \approx \vert A\vert \approx \Vert \mathbf{1}_{\varepsilon A}\Vert_{\mathfrak f_1^q}$. To show that $\mathcal B$ satisfies the Property (C), we use that the canonical basis in $\mathfrak f_1^q$ is unconditional: take $\textbf a\in\mathbb X$ and $\Lambda$ a greedy set of $\textbf{a}$, then
		$$\min_{n\in \Lambda}\vert a_n\vert \Vert \mathbf{1}_{\varepsilon \Lambda}\Vert \lesssim \min_{n\in \Lambda}\vert a_n\vert \Vert \mathbf{1}_{\varepsilon \Lambda}\Vert_{\mathfrak f_1^q}\lesssim \Vert \textbf a\Vert_{\mathfrak f_1^q}\leq \Vert \textbf a\Vert.$$
		Hence, the basis satisfies the Property (C). Also, since the basis is superdemocratic, using the Proposition \ref{propCS}, the basis satisfies the Property (A).
	\end{ex}
	
	\begin{defn}
		We say that $\mathcal B$ is \textbf{bidemocratic} if there exists a constant $C\geq 1$ such that
		$$\Vert \mathbf{1}_{\varepsilon A}\Vert \Vert \mathbf{1}_{\eta A}^*\Vert_*\leq C\vert A\vert,\; \forall\; \text{finite}\; A, \forall \varepsilon,\eta\in\lbrace\pm 1\rbrace.$$
Here, $\| \cdot \|_*$ is the norm of ${\mathbb{X}}^*$, and $\mathbf{1}_{\eta A}^* = \sum_{i\in A} \eta_i e_i^*$. 
	\end{defn}
	
	\begin{lem}
		If $\mathcal B$ is bidemocratic, then $\mathcal B$ satisfies the Property (C).
	\end{lem}
	\begin{proof}
		Here, we prove a stronger condition than Property (C). Take $x\in\mathbb X$ and $A\subset \supp(x)$. Then, taking $\eta = 1/\sgn\lbrace e_i^*(x)\rbrace$,
		\begin{eqnarray*}
			\min_{j\in A}\vert e_j^*(x)\vert \Vert \mathbf{1}_{\varepsilon A}\Vert \lesssim \min_{j\in A}\vert e_j^*(x)\vert \dfrac{\vert A\vert}{\Vert \mathbf{1}_{\eta A}^*\Vert_*}\leq \dfrac{\sum_{j\in A}\vert e_j^*(x)\vert}{\Vert \mathbf{1}_{\eta A}^*\Vert_*} \leq \dfrac{\mathbf{1}_{\eta A}^*(x)}{\Vert \mathbf{1}_{\eta A}^*\Vert_*}\leq \Vert x\Vert.
		\end{eqnarray*}
	\end{proof}
	
	\begin{cor}
% 		For $w\equiv 1$, 
All bidemocratic bases satisfy the ``classical'' Property (A).
	\end{cor}
	
	\begin{proof}
		If a basis $\mathcal{B}$ is bidemocratic, then it is superdemocratic.
%  and using the previous lemma and Proposition \ref{propCS}, we prove the result.
Now combine the preceding lemma with Proposition \ref{propCS}.
	\end{proof}

% 	Closed to the concept of the Property (C), we can talk about the following property:
Relaxing the assumptions of Definition \ref{d:prop_C}, we consider:
	
 	\begin{defn}\label{d:prop_D}
		We say that $\mathcal B$ satisfies the \textbf{Property (D)} if there exists a positive constant $C$ such that
		$$\min_{n\in A}\vert a_n\vert \Vert \mathbf{1}_A\Vert \leq C\Vert \sum_{n\in A}a_n e_n\Vert,$$
		for any finite set $A$ and scalars $(a_n)_{n\in A}$.
	\end{defn}
	It's clear that if $\mathcal B$ satisfies the Property (C), then $\mathcal B$ satisfies the Property (D) as well. 
	
	\begin{ex}\label{ex:no_D}[Example of superdemocratic basis in a Banach space without the Property (D)]
		Let $\mathbb X= \ell_1 \oplus c_0$ and $\Vert (x,y)\Vert = \Vert x\Vert_{\ell_1}+\Vert y\Vert_\infty$. Let $(e_n)_n$ be the canonical basis in $\ell_1$ and $(f_m)_m$ the canonical basis in $c_0$. We define
		
		$$E_{2n-1}=\left(\frac{1}{2}e_n,\frac{-1}{2}f_n\right),\;\; E_{2n}=\left(\frac{1}{4}e_n, \frac{3}{4}f_n\right), n=1,2,...,$$
		
		and consider $\mathcal B = \lbrace E_n\rbrace_n=\lbrace E_{2n-1},E_{2n}\rbrace_n$. This basis is normalized.
%		The biorthogonal dual $\mathcal B^*$ is given by 
%		$$E_{2n-1}^*(x,y)=\frac{3}{2}e_n^*(x)-\frac{1}{2}f_n^*(y),\; E_{2n}^*(x,y)=e_n^*(x)+f_n^*(y), n=1,2,...$$
%		This biorthogonal dual is a basis of $\mathbb X^* = \ell_\infty\oplus_\infty\ell_1$ and in this space it will be written as
%		$$E_{2n-1}^*=(\frac{3}{2}e_n^*, -\frac{1}{2}f_n^*),\; E_{2n}^*=(e_n^*,f_n^*), n=1,2,...$$
%		The basis $\mathcal B^*$ is also seminormalized with $\Vert E_{2n-1}^*\Vert_* = 3/2$ and $\Vert E_{2n}^*\Vert_* = 1$, $n=1,2,...$. 
		
		To prove that this basis is superdemocratic, we show the following proposition:
		\begin{prop}
			$D(m)\approx d(m)\approx m$, where $D(m):=\sup\lbrace \Vert \mathbf{1}_{\varepsilon A}\Vert: \vert A\vert\leq m, \varepsilon\in\lbrace\pm 1\rbrace\rbrace$ and $d(m):=\inf\lbrace \Vert \mathbf{1}_{\varepsilon A}\Vert: \vert A\vert\geq m, \varepsilon\in\lbrace\pm 1\rbrace\rbrace$.	
		\end{prop}
		\begin{proof}
			Of course, $d(m)\leq D(m)\leq m$. We prove that $d(m)\geq \frac{1}{8}m$. To this end, given $A\subset \mathbb N$ finite, we write
			\begin{eqnarray*}
				A_1 = \lbrace k \in\mathbb N : 2k\in A\; \text{and}\;2k-1\in A\rbrace,\\
				A_2 = \lbrace k \in\mathbb N : 2k\in A\; \text{and}\; 2k-1\not\in A\rbrace,\\
				A_3= \lbrace k \in\mathbb N : 2k\not\in A\; \text{and}\; 2k-1\in A\rbrace.
			\end{eqnarray*}
			Observe that the sets $A_1, A_2, A_3$ are mutually disjoint, and $2\vert A_1\vert + \vert A_2\vert + \vert A_3\vert = \vert A\vert$. For any choice of signs,
			\begin{eqnarray*}
				\Vert \mathbf{1}_{\varepsilon A}\Vert &=& \Vert \sum_{k\in A_1} \varepsilon_{2k}E_{2k} + \varepsilon_{2k-1}E_{2k-1}+\sum_{k\in A_2}\varepsilon_{2k}E_{2_k} + \sum_{k\in A_3}\varepsilon_{2k-1}E_{2k-1}\Vert\\
				&=&\Vert \sum_{k\in A_1}\left( [\frac{1}{4}\varepsilon_{2k}+\frac{1}{2}\varepsilon_{2k-1}]e_k, [\frac{3}{4}\varepsilon_{2k}-\frac{1}{2}\varepsilon_{2k-1}]f_k\right)\\
				&+&\sum_{k\in A_2}\varepsilon_{2k}(\frac{1}{4}e_k,\frac{3}{4}f_k)+\sum_{k\in A_3}\varepsilon_{2k-1}(\frac{1}{2}e_k,-\frac{1}{2}f_k)\Vert\\
				&\geq& \sum_{k\in A_1}\vert \frac{1}{4}\varepsilon_{2k}+\frac{1}{2}\varepsilon_{2k-1}\vert + \sum_{k\in A_2}\frac{1}{4}+\sum_{k\in A_3}\frac{1}{2}.	
			\end{eqnarray*}
			Therefore, $$\Vert \mathbf{1}_{\varepsilon A}\Vert \geq \frac{1}{4}\vert A_1\vert + \frac{1}{4}\vert A_2\vert + \frac{1}{2}\vert A_3\vert \geq \frac{1}{8}\vert A\vert.$$
			This finishes the proof.
		\end{proof}
		
		Back to Example \ref{ex:no_D}:
%		using a similar argument, it's possible to show that $\mathcal B^*$ is also superdemocratic.
		to see that the basis does not have the Property (D), take $z=\sum_{n=1}^N 2E_{2n}-\sum_{n=1}^N E_{2n-1}$. Then,
		$$\Vert z\Vert = \Vert \sum_{n=1}^N (0,2f_n)\Vert = 2.$$
		Write $z=\sum_{i\in A}a_i E_i$. Then, $\min_{i\in A}\vert a_i\vert = 1$ and 
		$$\Vert \sum_{i\in A} E_i\Vert = \Vert \sum_{n=1}^N E_{2n}+\sum_{n=1}^N E_{2n-1}\Vert = \Vert \sum_{n=1}^N (\frac{3}{4}e_n,\frac{1}{4}f_n)\Vert = \frac{3}{4}N+\frac{1}{4}.$$
		This shows that the Property (D) fails.
	\end{ex}
	
	% Also, in \cite{DKTW}, we can find the Lemma 3.13 that establish the following:
	Lemma 4.13 of \cite{DKTW} establishes that,
	if $w$ is equivalent to the constant, and $\mathcal B$ is $w$-semi-greedy,
	then $\mathcal B$ satisfies the Property (D). Here, we improve this result showing
	that the condition of being $w$-semi-greedy implies the Property (C).
	
	\begin{prop}\label{p:SG_implies_C}
		Assume that $w$ is equivalent to the constant, and the basis
		$\mathcal B$ is $w$-semi-greedy. Then $\mathcal B$ satisfies the Property (C).
	\end{prop}
	
	\begin{proof}
By Theorem \ref{t:semi_greedy}, $\mathcal B$ has the $w$-Property (A). By Proposition \ref{w-equivalent},
$\mathcal B$ also has the ``classical'' Property (A). This, in turn, implies the Property (C).
	\end{proof}

	% To finish this section, based on a Proposition 3.5 of \cite{DKTW},
	% we assume some conditions to study when the basis $\mathcal B$
	% is equivalent to the unit vector basis of $c_0$. Moreover, this result
	% is an improvement of the Proposition 3.5 of \cite{DKTW} because
	% we only assume that our basis satisfies the $w$-Property (A).

	\section{$w$-Partially-greedy bases}\label{s:partially_greedy}
	Partially-greedy and conservative bases were introduced in \cite{DKKT}, in order to
	compare the errors of greedy approximation with those of the canonical approximation
	relative to Schauder basis (the ``tails'' of the basis expansion).
	In this section we define $w$-partially-greedy and $w$-conservative bases
	and extend the characterization of partially-greedy bases proved in \cite{DKKT} to this more general setting.
	
	\begin{defn}\label{new-wpg}
		We say that $\mathcal B$ is \textbf{$w$-partially-greedy} if for all $m$ and $r$ such that $w(\lbrace 1,...,m\rbrace )\leq w(A_r(x))$, there exists a positive constant such that
		\begin{eqnarray}\label{cp}
		\Vert x-\mathcal{G}_r(x)\Vert \leq C\Vert \sum_{n=m+1}^\infty e_i^*(x)e_i\Vert.
		\end{eqnarray}
		We denote by $C_p$ the least constant that satisfies \eqref{cp} and we say that $\mathcal B$ is $C_p$-$w$-partially-greedy.
	\end{defn}
	
	\begin{defn}
		We say that $\mathcal B$ is \textbf{$w$-conservative} if there exists a positive constant $C$ such that
		\begin{eqnarray}\label{cco}
		\Vert \mathbf{1}_A\Vert \leq C\Vert \mathbf{1}_B\Vert,
		\end{eqnarray}
		for all pair of $A,B\in\mathbb N^{<\infty}$ such that $A<B$ and $w(A)\leq w(B)$.
		We denote by $C_c$ the least constant that satisfies \eqref{cco} and we say that $\mathcal B$ is $C_c$-$w$-conservative.
	\end{defn}
	
	\begin{rem}
		If $w\equiv 1$, we recover the classical definition of partially-greediness (resp.~conservativeness), and we will say that $\mathcal B$ is partially-greedy (resp.~conservative).
	\end{rem}
	
{\begin{rem}Note that for some choices of weight $w$, the property of $w$-conservativeness can be in some sense trivial. For instance, if $w=(2^{-n})_{n=1}^\infty$ then {\it every} seminormalized basis is $w$-conservative. This is because there are no nonempty $A,B\in\mathbb{N}^{<\infty}$ satisfying both $A<B$ and $w(A)\leq w(B)$.\end{rem}
	
	Let us give a simple characterization % of this kind of triviality.
	of weights for which this occurs.
	
	\begin{prop}Let $w$ be a weight and set
		\begin{multline*}s_w:=\sup\big\{n\in\mathbb{N}_0:\text{there exist }A\in\mathbb{N}^n\text{ and }B\in\mathbb{N}^{<\infty}\text{ such that }A<B\text{ and }w(A)\leq w(B)\big\}.\end{multline*}
		Then $s_w<\infty$ if and only if every seminormalized basis is $w$-conservative.\end{prop}
	
	\begin{proof}($\Longrightarrow$): Suppose $s_w<\infty$. Let $(e_n)_{n=1}^\infty$ be a seminormalized basis for a Banach space $\mathbb X$, and select $A,B\in\mathbb{N}^{<\infty}$ such that $A<B$ and $w(A)\leq w(B)$. Observe that $\Vert\mathbf{1}_A\Vert\leq c_2\cdot\vert A\vert\leq c_2s_w$. It follows immediately that $\Vert\mathbf{1}_A\Vert\leq c_2s_w\leq c_2^2s_w\Vert\mathbf{1}_B\Vert$. Hence, $(e_n)_{n=1}^\infty$ is $(c_2^2s_w)$-$w$-conservative.
		
		($\Longleftarrow$): Suppose $s_w=\infty$. Let's inductively construct sequences $(A_n)_{n=1}^\infty\subset\mathbb{N}^{<\infty}$ and $(B_n)_{n=1}^\infty\subset\mathbb{N}^{<\infty}$ satisfying
		$$A_1<B_1<A_2<B_2<A_3<B_3<\ldots,$$
		and also satisfying $\vert A_n\vert \geq n$ and $w(A_n)\leq w(B_n)$ for all $n\in\mathbb{N}$.
		Let us begin by selecting $A_1\in\mathbb{N}^{<\infty}$ and $B_1\in\mathbb{N}^{<\infty}$ with $\vert A_1\vert=1$, $A_1<B_1$, and $w(A_1)\leq w(B_1)$, which is possible as $s_w\geq 1$. This is the base case; from now on, we proceed inductively. Since $s_w=\infty$, we may select $\widehat{A}_{n+1}\in\mathbb{N}^{<\infty}$ and $B_{n+1}\in\mathbb{N}^{<\infty}$ with $\vert \widehat{A}_{n+1}\vert>n+\max B_n$, $\widehat{A}_{n+1}<B_{n+1}$, and $w(\widehat{A}_{n+1})<w(B_{n+1})$. Now set $A_{n+1}=\widehat{A}_{n+1}\setminus\{1,\ldots,\max B_n\}$ so that we have $\vert A_{n+1}\vert>n$, $A_{n+1}<B_{n+1}$, and $w(A_{n+1})<w(B_{n+1})$. This completes the inductive step, and gives us our intertwining sequences with the desired properties. We may now define a norm on $c_{00}$ via the rule
		$$\Vert(a_n)_{n=1}^\infty\Vert_\mathbb X=\Vert(a_n)_{n=1}^\infty\Vert_\infty\vee\sup_{k\in\mathbb{N}}\sum_{n\in A_k}|a_n|\;\;\;\forall\,(a_n)_{n=1}^\infty\in c_{00},$$
		and denote by $\mathbb X$ the completion of $c_{00}$ under this norm. It is clear that the standard canonical basis for this space form a normalized 1-unconditional basis. However, it fails to be $w$-conservative as $\Vert\mathbf{1}_{A_k}\Vert_\mathbb X=\vert A_k\vert\geq k$ whereas $\Vert\mathbf{1}_{B_k}\Vert_\mathbb X=1$ for all $k\in\mathbb{N}$.\end{proof}
	
	\begin{prop}Let $w$ be a nonincreasing weight, i.e., $w_{n+1}\leq w_n$ for all $n\in\mathbb{N}$. Then every conservative basis in a Banach space is $w$-conservative with the same constant.\end{prop}
	
	\begin{proof}Let $(e_n)_{n=1}^\infty$ be a conservative basis in a Banach space $\mathbb X$, and select any $A,B\in\mathbb{N}^{<\infty}$ satisfying both $A<B$ and $w(A)\leq w(B)$. Now,
		$$\vert A\vert\cdot w_{\max A}\leq w(A)\leq w(B)\leq\vert B\vert\cdot w_{\min B}\leq \vert B\vert\cdot w_{\max A},$$
		so that $\vert A\vert\leq\vert B\vert$.
	\end{proof}
	
% 	\begin{rem}If $(e_n)_{n=1}^\infty$ is a conservative basis instead of just a conservative sequence, then due to Remark \ref{cons-equiv} it is enough that $w$ be  equivalent to a nonincreasing weight, in which case $(e_n)_{n=1}^\infty$ will be $w$-conservative for all such weights.\end{rem}
	
	\begin{thm}\label{t:partially_greedy}
		A basis $\mathcal B$ is $w$-partially-greedy if and only if $\mathcal B$ is quasi-greedy and $w$-conservative.
	\end{thm}
	
	\begin{proof}
		Assume that $\mathcal B$ is $C_p$-$w$-partially-greedy.
		\begin{enumerate}
			\item \underline{$w$-conservative:} take $A$ and $B$ such that $A<B$ and $w(A)\leq w(B)$. Let $m= \max A$ and define the set $D=[1,..,m]\setminus A$. Of course,
			$$w(\lbrace 1,...,m\rbrace)=w(A\cup D)\leq w(B\cup D).$$
			Define now $x:=\mathbf{1}_A+(1+\delta)\mathbf{1}_{B\cup D}$. Then,
			$$\Vert \mathbf{1}_A\Vert = \Vert x-\mathcal{G}_{\vert B\cup D\vert}(x)\Vert \leq C_p\Vert (1+\delta)\mathbf{1}_B\Vert.$$
			Taking $\delta\rightarrow 0$, the basis is $w$-conservative.
			
			\item\underline{Quasi-greedy:} here, we consider two cases.
			\begin{enumerate}
				\item[a)] Assume that the index $1 \not\in A_r(x)$. Define then
				$\tilde{x}=te_1+\sum_{i=2}^\infty e_i^*(x)e_i = x + (t - e_1^*(x)) e_1$,
				with $t=\max\vert e_i^*(x)\vert + \delta$ with $\delta>0$. Then
				$$\mathcal{G}_r(\tilde{x})= te_1+\mathcal{G}_{r-1}(x),$$
				hence $\tilde{x}-\mathcal{G}_r(\tilde{x}) = \sum_{i=2}^\infty e_i^*(x)e_i - \mathcal{G}_{r-1}(x).$ Thus, using the triangle inequality and the fact that $w(\lbrace 1\rbrace)\leq w(A_r(\tilde{x}))$,
				\begin{eqnarray*}
					\Vert \mathcal{G}_{r-1}(x)\Vert &\leq& \Vert \tilde{x}-\mathcal{G}_r(\tilde{x})\Vert + \Vert \sum_{i=2}^\infty e_i^*(x)e_i \Vert \leq C_p\Vert \sum_{i=2}^\infty e_i^*(x)e_i \Vert + \Vert \sum_{i=2}^\infty e_i^*(x)e_i \Vert\\\nonumber
					&\leq& (C_p+1)(1+K_b)\Vert x\Vert.
				\end{eqnarray*}
				That's implies that $\Vert \mathcal{G}_r(x)\Vert \leq ((C_p+1)(K_b+1)+c_2^2)\Vert x\Vert.$
				\item[b)] Assume now that $1 \in A_r(x)$. Taking the same $\tilde{x}$ that in the above case,
				$$\mathcal{G}_r(\tilde{x}) = \mathcal{G}_{r-1}(x-e_1^*(x)e_1)+te_1,$$
				so $\tilde{x}-\mathcal{G}_r(\tilde{x}) = \sum_{i=2}^\infty e_i^*(x)e_i -\mathcal{G}_{r-1}(x-e_1^*(x)e_1).$ Hence, using the same argument than before,
				\begin{eqnarray*}
					\Vert \mathcal{G}_{r-1}(x-e_1^*(x)e_1)\Vert &\leq& \Vert \tilde{x}-\mathcal{G}_r(\tilde{x})\Vert + \Vert \sum_{i=2}^\infty e_i^*(x)e_i \Vert\\
					&\leq& C_p\Vert \sum_{i=2}^\infty e_i^*(x)e_i \Vert + \Vert \sum_{i=2}^\infty e_i^*(x)e_i \Vert\\\nonumber
					&\leq& (C_p+1)(1+K_b)\Vert x\Vert.
				\end{eqnarray*}
				
				Now, $\Vert \mathcal{G}_{r-1}(x-e_1^*(x)e_1)\Vert = \Vert \mathcal{G}_{r-1}(x-e_1^*(x)e_1)+e_1^*(x)e_1-e_1^*(x)e_1\Vert=\Vert \mathcal G_r(x)-e_1^*(x)e_1\Vert$. So, using the triangle inequality, we obtain that $\Vert \mathcal G_r(x)\Vert \leq ((C_p+1)(K_b+1)+c_2^2)\Vert x\Vert$.
			\end{enumerate}
		\end{enumerate}
		
		Now, assume that $\mathcal B$ is $C_c$-$w$-conservative and $C_q$-quasi-greedy, and show that $\mathcal B$ is $w$-partially-greedy.
		Take $x\in\mathbb X$, $m$, and $r$ as in the definition of $w$-partially-greedy,
		and consider the sets
		$$D:=\lbrace \rho(j): j\leq r, \rho(j)\leq m\rbrace, \, \, 
		B:=\lbrace \rho(j): j\leq r, \rho(j)>m\rbrace, \, \,
		A:= [1,...,m]\setminus D,$$
		where $\rho$ is the greedy ordering.
		Then $A_r(x) = B \cup D$, and
		$w(A)=w(\lbrace 1,...,m\rbrace)-w(D)\leq w(A_r(x))-w(D)=w(B).$ 
		$$x-\mathcal{G}_r(x) = \sum_{i=m+1}^\infty e_i^*(x)e_i - P_B(x) + P_A(x).$$
		On the one hand, $\Vert P_B(x)\Vert \leq 2C_q\Vert \sum_{i=m+1}^\infty e_i^*(x)e_i\Vert$. On the other hand, using Lemmas \ref{trun} and \ref{part1} with $\eta \equiv \sgn(e_j^*(x))$,
		\begin{eqnarray*}
			\Vert P_A(x)\Vert &\leq& 4C_qC_c\max_A\vert e_i^*(x)\vert\Vert \mathbf{1}_{\eta B}\Vert \leq 4C_qC_c\min_B\vert e_i^*(x)\vert\Vert \mathbf{1}_{\eta B}\Vert\\
			&\leq& 8C_q^2C_c\Vert P_B(x)\Vert \leq 8C_q^3C_c\Vert \sum_{i=m+1}^\infty e_i^*(x)e_i\Vert.
		\end{eqnarray*}
		Then, $\Vert x-\mathcal{G}_r(x)\Vert \lesssim C_q^3C_c\Vert \sum_{i=m+1}^\infty e_i^*(x)e_i\Vert.$
	\end{proof}

	\begin{rem} Note that if the inequality $\Vert x-\mathcal G_r(x)\Vert \le C\Vert x-S_m(x)\Vert$
		is satisfied for $m$ and $r$, then it is automatically satisfied -- with a different constant --
		for any $n < m$ and the same $r$
		(since $C\Vert x-S_m(x)\Vert \le (1 + K_b)C\Vert x-S_n(x)\Vert$ where $K_b$ is the basis constant).
		So we only need to check the condition in the definition of $w$-partially-greedy for the largest $m$
		satisfying $w([1,\dots,m]) \le w(A_r(x))$. 
		
		Using the constant weight $w\equiv 1$, we recover the usual definition of a partially-greedy basis.
		Indeed, for $w\equiv 1$, the largest $m$ satisfying the definition is $m = r$,
		which recaptures the original definition of partially-greedy given in \cite{DKKT}.
	\end{rem}

	\subsection{Example of conservative and not democratic basis.}
	Define the set $$\mathcal S=\lbrace A\in\mathbb N^{<\infty} : \vert A\vert \leq \sqrt{\min A}\rbrace.$$
	Observe that $\mathcal S$ has the \textit{spreading property}, i.e, if $m\in\mathbb N$, $(f_i)_{i=1}^m\in\mathcal S$ and $(g_i)_{i=1}^m\in\mathbb N^m$ with $f_i\leq g_i$ for all $i=1,\dots,n$, then $(g_i)_{i=1}^m \in\mathcal S$. It also \textit{hereditary}, i.e., if $A\in\mathcal S$ and $B \subset A$ then $B\in \mathcal S$.
	
	Now, let $\mathbb X$ be the Banach space that we define like the completion of $c_{00}$ under the norm
	$$\Vert (a_n)_n\Vert = \sup_{A\in \mathcal S}\sum_{n\in A}\vert a_n\vert.$$
	Observe that this is a very slight modification of the Schreier space.
	
	Of course, the canonical basis $(e_n)_n$ is a normalized 1-unconditional basis. Note that the hereditary property guarantees that
	$$\Vert \mathbf{1}_A\Vert = \sup_{F\in\mathcal S, F\subseteq A}\vert F\vert.$$
	
	Now, if $A<B$ and $\vert A\vert \leq \vert B\vert$, then there is $F\in\mathcal S$ with $F\subseteq A$ such that $\Vert \mathbf{1}_A\Vert = \vert F\vert$. By the spreading property, we can ``push out" $F$ to obtain a set $G\subseteq B$ such that $G\in\mathcal S$ and $\vert G\vert= \vert F\vert$. Hence,
	$$\Vert \mathbf{1}_A\Vert = \vert F\vert =\vert G\vert \leq \Vert \mathbf{1}_B\Vert.$$
	Thus, the basis is conservative with constant 1.
	
	To prove that the basis is not democratic, we can select the sets $A=\lbrace N^2+1,...,N^2+N\rbrace$ and $B=\lbrace 1,...,N\rbrace$ . Then, since $A\in \mathcal S$, $\Vert \mathbf{1}_A\Vert = N$. However, $\Vert \mathbf{1}_B\Vert \leq \sqrt{N}$, hence the basis is not democratic: to prove this upper estimate, take a set $A_1 \in \mathcal S$ such that $\Vert \mathbf{1}_B\Vert = \vert A_1\vert$. Then, $\min A_1\leq N$, so $\vert A_1\vert \leq \sqrt{N}$. Hence, $\Vert \mathbf{1}_B\Vert \leq \sqrt{N}$.
	
	\begin{rem}
	Of course, since the canonical basis is unconditional (hence, quasi-greedy) and conservative, is partially-greedy, but not almost-greedy because is not democratic.
	\end{rem}
	\subsection{Example of a $w$-greedy basis which is not conservative (hence not greedy)} % in a Banach space}
	
	\begin{defn}Fix $1\leq p<q\leq\infty$, and consider $(e_n)_{n=1}^\infty$ and $(f_n)_{n=1}^\infty$ the respective canonical bases of $\ell_p$ and $\ell_q$ (or $c_0$ if $q=\infty$). Let $w=(w_n)_{n=1}^\infty\in(0,\infty)^\mathbb{N}$. We define the {\bf Rosenthal-Woo space} $X_{q,p,w}$ as the closed subspace $[f_n\oplus w_ne_n]_{n=1}^\infty$ of $\ell_q\oplus_\infty\ell_p$. For $\textbf{s}=(s_n)_{n=1}^\infty$, the summing basis of $c_0$, we can define $X_{q,\textbf{s},w}$ similarly as the subspace $[f_n\oplus w_ns_n]_{n=1}^\infty$ of $\ell_q\oplus_\infty c_0$.\end{defn}
	
	It was mentioned in \cite{DKTW} that if $w\in(0,\infty)^\mathbb{N}$ satisfies $w\in c_0\setminus\ell_1$ then the basis formed by completing $c_{00}$ under the norm
	$$\Vert(a_n)_{n=1}^\infty\Vert_\infty\vee\left(\sum_{n=1}^\infty|a_n|^2w_n\right)^{1/2},\;\;\;(a_n)_{n=1}^\infty\in c_{00},$$
	forms a normalized 1-$w$-greedy basis which is not greedy. In fact, this is just the canonical basis of the Rosenthal-Woo space $X_{\infty,2,w^{1/2}}$. More generally, we have the following.

	\begin{prop}\label{RW-infty}Fix $1\leq p<\infty$ and $w\in(0,\infty)^\mathbb{N}\cap(c_0\setminus\ell_1)$. Then the canonical basis of $X_{\infty,p,w^{1/p}}$ is $1$-$w$-greedy, but it is not conservative.\end{prop}
	
	\begin{proof} Clearly it is unconditional with constant 1. To prove that the canonical basis is $w$-greedy
		with constant 1, we need to show that it satisfies the $w$-Property (A) with constant 1 (Theorem \ref{1wgreedy}).
		For that, take $x\in\mathbb X_{\infty,p,w^{1/p}}$ with $\sup_j\vert e_j^*(x)\vert \leq 1$,
		and consider $A,B \subset \N$ such that $A\cap B = \emptyset$, $\supp(x)\cap (A\cup B)=\emptyset$, $w(A)\leq w(B)$. Then, if $\varepsilon$ and $\eta$ are arbitrary choice of signs, 
		\begin{eqnarray*}
			\Vert x+\mathbf{1}_{\varepsilon A}\Vert &=& 1 \vee \left(\sum_{n\in \supp(x)}\vert e_n^*(x)\vert^p w_n + w(A)\right)^{1/p}\\
			&\leq& 1 \vee \left(\sum_{n\in \supp(x)}\vert e_n^*(x)\vert^p w_n + w(B)\right)^{1/p}=\Vert x+\mathbf{1}_{\eta B}\Vert.
		\end{eqnarray*}
		Then, the basis satisfies the $w$-Property (A) with constant 1, hence, using that the basis is unconditional with constant 1, the basis is $w$-greedy with constant 1.
		
		To see that it fails to be conservative, fix $m\in\mathbb{N}$ and set $A_m=\{1,\ldots,m\}$ and $B_{m,k}=\{k+1,\ldots,k+m\}$ for each $k\in\mathbb{N}$. Now observe that $w\in c_0$ ensures that $w(B_{m,k})\to 0$ when $k\rightarrow\infty$ and hence $\Vert\mathbf{1}_{B_{m,k}}\Vert_{\infty,p,w^{1/p}}=1$ for sufficiently large $k$. Hence, we may select $k_m\in\mathbb{N}$ so that $B_{m,k_m}>A_m$ and $\Vert\mathbf{1}_{B_{m,k_m}}\Vert_{\infty,p,w^{1/p}}=1$. On the other hand, $w\notin\ell_1$ guarantees that $w(A_m)\to\infty$.\end{proof}
	
	\begin{rem}The above proof works for $X_{\infty,\textbf{s},w}$ as well, except that the basis is no longer unconditional. It yields an example of a subspace of $c_0$ with a basis which is $w$-democratic but not conservative, so long as $w\in(0,\infty)^\mathbb{N}\cap(c_0\setminus\ell_1)$.\end{rem}

	However, the situation is different for the canonical basis of $X_{q,p,w^{1/p}}$ when $q\neq\infty$. These spaces fail to contain any copies of $c_0$, and hence do not admit $v$-greedy bases for $v$ non-seminormalized. Even when $v$ is seminormalized the canonical basis may not be $v$-democratic (nor $v$-greedy), as we will see momentarily.
	
	\begin{prop}\label{p:w-conservative}
		Fix $1\leq p<q<\infty$, and let $w\in(0,\infty)^\mathbb{N}$ be decreasing. % essentially decreasing.
		Then the canonical basis of $X_{q,p,w^{1/p}}$ is unconditional and $w$-conservative with constants 1.\end{prop}

	\begin{proof}% By Proposition \ref{w-equivalent} we may assume $w$ is nonincreasing.
 Select $A,B\in\mathbb{N}^{<\infty}$ with $w(A)\leq w(B)$ and $A<B$. Since $w$ is decreasing, we must have
		$|A| \leq |B|$. Thus,
		$$\Vert\mathbf{1}_A\Vert_{q,p,w^{1/p}}=(\vert A\vert)^{1/q}\vee w(A)^{1/p}\leq(\vert B\vert)^{1/q}\vee w(B)^{1/p}=\Vert\mathbf{1}_{B}\Vert_{q,p,w^{1/p}}.$$\end{proof}

	\begin{prop}\label{not-conservative}Fix $1\leq p<q<\infty$ and $0<\theta<1-\frac{p}{q}$. 
		Let $w=(w_n)_{n=1}^\infty\in(0,\infty)^\mathbb{N}$ be defined by $w_n=n^{-\theta}$ for $n\in\mathbb{N}$.
		Then the canonical basis for $X_{q,p,w^{1/p}}$ is not conservative, and not $w$-democratic.\end{prop}
	
	\begin{proof}
		First establish that our basis is not conservative. As in the proof of Proposition \ref{RW-infty},
		for $k,m\in\mathbb{N}$ we set $A_m=\{1,\ldots,m\}$ and $B_{m,k}=\{k+1,\ldots,k+m\}$, and for each $m\in\mathbb{N}$ we find $k_m\in\mathbb{N}$ large enough that
		$A_m<B_{m,k_m}$ and $\Vert\mathbf{1}_{B_{m,k_m}}\Vert_{q,p,w^{1/p}}=m^{1/q}$. Meanwhile,
		$$\Vert\mathbf{1}_{A_m}\Vert_{q,p,w^{1/p}}
		\geq w(A_m)^{1/p}
		=\left(\sum_{n=1}^m n^{-\theta}\right)^{1/p}
		\geq\left(\int_1^m t^{-\theta}\;dt\right)^{1/p}
		=\left(\frac{m^{1-\theta}-1}{1-\theta}\right)^{1/p}$$
		so that, due to $1-\theta-p/q>0$,
		$$\frac{\Vert\mathbf{1}_{A_m}\Vert_{q,p,w^{1/p}}}{\Vert\mathbf{1}_{B_{m,k_m}}\Vert_{q,p,w^{1/p}}}
		\geq\left(\frac{m^{1-\theta}-1}{1-\theta}\right)^{1/p}m^{-1/q}
		=\left(\frac{m^{1-\theta-p/q}-m^{-p/q}}{1-\theta}\right)^{1/p}\to\infty.$$
		
		We next sketch the proof of the lack of $w$-democracy. To this end, consider the sets $A_n$ and
		$B_{m,k}$ as defined in the preceding paragraph. There exist universal constants $c$ and $C$ so that
		$w(A_n) \geq c n^{1-\theta}$, and $\Vert\mathbf{1}_{A_n}\Vert \leq C n^{1-\theta}$.
		% Find $m$ and $k$ so that ??? 
		Then $w(B_{m,k}) < m k^{-\theta}$, while $\Vert\mathbf{1}_{B_{m,k}}\Vert \geq m^{1/q}$.
		% Note that $\Vert\mathbf{1}_{B_{m,k}}\Vert \geq m^{1/q}/w({B_{m,k}}) k^\theta m^{1/q-1}$.
		% A simple computation shows that, for any $n$, one can find $m$ and $k$ for which
		For large values of $k$, select $m \in [cn^{1-\theta}k^\theta/2, cn^{1-\theta}k^\theta]$.
		Then $w(B_{m,k}) \leq w(A_n)$, yet the ratio
		$$
		\frac{\Vert\mathbf{1}_{B_{m,k}}\Vert}{\Vert\mathbf{1}_{A_n}\Vert} \approx
		\frac{m^{1/q}}{n^{1-\theta}} \approx \frac{n^{(1-\theta)/q} k^{\theta/q}}{n^{1-\theta}} =
		k^{\theta/q} n^{-(1-\theta)(1-1/q)}
		$$
		can be arbitrarily large (for large $k$), ruling out the possibility of $w$-democracy.
	\end{proof}
	
% 	Since $\ell_p\oplus\ell_q$ contains no copy of $c_0$, % the following is now immediate.
% Proposition \ref{p:find c0} immediately implies:
	
	\begin{cor}Fix $1\leq p<q<\infty$ and $0<\theta<1-\frac{p}{q}$. Let $w=(w_n)_{n=1}^\infty\in(0,\infty)^\mathbb{N}$ be defined by $w_n=n^{-\theta}$ for $n\in\mathbb{N}$. Then the canonical basis for $X_{q,p,w^{1/p}}$ is not $v$-democratic for any weight $v\in(0,\infty)^\mathbb{N}$.\end{cor}
	
\begin{proof}
Suppose, for the sake of contradiction, that the canonical basis for $X_{q,p,w^{1/p}}$ is $v$-democratic.
Note that this basis contains no subsequences equivalent to the $c_0$-basis.
Proposition \ref{p:find c0} shows that $0 < \inf v_n \leq \sup v_n < \infty$ -- that is,
the weight $v$ is equivalent to a constant. Then, by Remark \ref{cons-equiv},
the canonical basis for $X_{q,p,w^{1/p}}$ has to be democratic, hence conservative.
This, however, contradicts Proposition \ref{not-conservative}.
\end{proof}

\begin{rem}Consider again the weight $w$ from Proposition \ref{not-conservative}.
		\begin{itemize}\item
			% We now have an example of a space which admits a $w$-partially-greedy basis but not a $w$-almost-greedy one.
			It follows from Propositions \ref{p:w-conservative} and \ref{not-conservative} that
			the canonical basis of $X_{q,p,w^{1/p}}$ is $w$-partially-greedy basis, but not $w$-almost-greedy.
			However, the space $X_{q,p,w^{1/p}}$ does have an almost-greedy basis. Indeed, we recall from \cite[Theorem 10.7.1]{AK16} that if $\mathbb X$ has a complemented subspace with a symmetric basis and finite cotype then $\mathbb X$ admits an almost-greedy basis. If $w\in(0,\infty)^\mathbb{N}\cap(c_0\setminus\ell_{(pq)/(q-p)})$ then the Woo-Rosenthal spaces $X_{q,p,w}$ contain complemented copies of $\ell_p$ and $\ell_q$ (or $c_0$ if $q=\infty$; see \cite[Corollary 3.2]{Woo75}), and hence satisfy this condition.

\item Just as in Proposition \ref{not-conservative}, one can show that
the canonical basis of $X_{q,\textbf{s},w}$ is not conservative when $0<\theta<1-\frac{1}{q}$.
However, it is not quasi-greedy, either. For $\theta=1-\frac{1}{q}$, this basis
becomes quasi-greedy and democratic (for $q=2$, this was observed in \cite[Example 10.2.9]{AK16},
the argument is valid for all $1<q<\infty$).
\end{itemize}\end{rem}

	\section{Questions}\label{s:questions}
	\begin{itemize}
		\item Does Property (D) imply Property (C)?
		\item Does Property (A) imply Property (D)?
		% If not, find an example with the Property (A) and not with the Property (D).
		\item If a basis has Properties (C) and (A), is it necessarily semi-greedy?
		\item Is it possible to formulate a new property so that every conservative basis
		with this property is necessarily democratic? % $\Rightarrow$ democracy?
		\item If $w=(1,1,...)$ in Theorem \ref*{t:partially_greedy}, do we get the same constant as in the classic case (\cite[Theorem 3.4]{DKKT})?
	\end{itemize}
	
	\section{Appendix}\label{appendix}
	The purpose of this appendix is to show two basic lemmas.
 The first one resembles a result from \cite{BBG}. 
	For each $\lambda>0$, we define the $\lambda$-truncation of $z\in\mathbb C$ by
	\begin{equation*}
	T_\lambda(z) = \left\lbrace
	\begin{array}{ll}
	\lambda \sgn (z) & \textup{if } \vert z\vert \geq \lambda\\
	z & \textup{if } \vert z\vert \leq \lambda
	\end{array}
	\right.
	\end{equation*}
	
	We extend $T_\lambda$ to an operator on $\mathbb X$ by 
	$$T_\lambda(x)=\sum_{j}T_\lambda(e_j^*(x))e_j = \sum_{j\in \Lambda_\lambda} \lambda \sgn(e_j^*(x))e_n + \sum_{j\in \Lambda_\lambda^c}e_j^*(x)e_j,$$
	where $\Lambda_\lambda = \lbrace j : \lambda < \vert e_j^*(x)\vert\rbrace$.
	
	\begin{lem}\label{trun}
		For all $\lambda>0$ and $x\in\mathbb{X}$, if $\mathcal B$ is $C_q$-quasi-greedy, we have
		$$\Vert T_\lambda (x)\Vert \leq C_q\Vert x\Vert,\;\; \Vert (I-T_\lambda)(x)\Vert \leq (C_q+1)\Vert x\Vert,\;\; \alpha\Vert \mathbf{1}_{\varepsilon \Lambda}\Vert\leq 2C_q\Vert x\Vert,$$
		where $\alpha = \min_{j\in\Lambda}\vert e_j^*(x)\vert $, $\Lambda$ is a greedy set of $x$ and $\varepsilon \equiv \sgn(e_j^*(x))$.
		
		Moreover, if $\mathcal B$ is $K_u$-unconditional, for every set $A$ with $\vert A\vert<\infty$, $$\Vert T_\lambda (I-P_A)(x)\Vert \leq K_u\Vert x\Vert.$$
	\end{lem}
	
	\begin{proof}
		$$T_\lambda(x)=\int_{0}^{1}\left[\sum_j \chi_{[0,\frac{\lambda}{\vert e_j^*(x)\vert}]}(s)e_j^*(x)e_j \right]ds = \int_{0}^{1} (I-P_{\Lambda_{\lambda,s}})x\,ds,$$
		where $\Lambda_{\lambda,s}=\lbrace j : \frac{\lambda}{s}<\vert e_j^*(x)\vert\rbrace$ is a greedy set of $x$ of finite cardinality. Then, using the Minkowski's integral inequality,
		$$\Vert T_\lambda (x)\Vert \leq \int_{0}^{1}\Vert (I-P_{\Lambda_{\lambda,s}})x\Vert ds\leq C_q\Vert x\Vert.$$
		
		Also, since $(I-T_\lambda)x= \int_{0}^{1}P_{\Lambda_{\lambda,s}}(x)ds$, hence
		$$\Vert (I-T_\lambda)(x)\Vert \leq (C_q+1)\Vert x\Vert.$$
		
		Now, since $\alpha \mathbf{1}_{\varepsilon \Lambda} = T_{\alpha}(x)-P_{\Lambda^c}(x) = \int_{0}^{1}(P_{\Lambda}(x)-P_{\Lambda_{\alpha,s}}(x))ds$,
		$$\alpha\Vert \mathbf{1}_{\varepsilon \Lambda}\Vert \leq 2C_q\Vert x\Vert.$$
		On the other hand, if $A$ is a general set with $\vert A\vert<\infty$,
		$$T_\lambda (I-P_A)x = \int_{0}^{1}(I-P_{\Lambda_{\lambda,s}})(I-P_A)x\,ds = \int_{0}^{1}(I-P_{A\cup\Lambda_{\lambda,s}})x \,ds,$$
		thus $$\Vert T_\lambda (I-P_A)(x)\Vert \leq K_u\Vert x\Vert.$$
	\end{proof}
	
	The second lemma involves the concept of $w$-partially-greedy bases.
	\begin{lem}\label{part1}
		If $\mathcal B$ is $C_w$-$w$-conservative and $C_q$-quasi-greedy, then
		$$\Vert \sum_{j\in A}a_j e_j\Vert \leq 4C_qC_w\max_{j\in A}\vert a_j\vert\Vert \mathbf{1}_{\eta B}\Vert,$$
		for any sign $\eta$ and $A,B\in\mathbb N^{<\infty}$ such that $w(A)\leq w(B)$, $A<B$ and any collection of scalars $(a_j)_{j\in A}$.
	\end{lem}
	\begin{proof}
		We prove that $\Vert \mathbf{1}_{\varepsilon A}\Vert \leq 4C_qC_w\Vert \mathbf{1}_{\eta B}\Vert$ for any signs $\varepsilon$ and $\eta$. First, we can decompose $\textbf{1}_{\varepsilon A}=\textbf{1}_{A^+}-\textbf{1}_{A^-}$, where $A^{\pm}=\lbrace j\in A : \varepsilon_j = \pm 1\rbrace$. Then,
		$$\Vert \textbf{1}_{\varepsilon A}\Vert \leq \Vert \textbf{1}_{A^+}\Vert + \Vert \textbf{1}_{A^-}\Vert \leq 2C_w\Vert \textbf{1}_{B}\Vert.$$
		
		Now, using the condition to be quasi-greedy, it is clear that $\Vert \textbf{1}_B\Vert \leq 2C_q\Vert \textbf{1}_{\eta B}\Vert$, then
		$$\Vert \textbf{1}_{\varepsilon A}\Vert \leq 4C_qC_w\Vert \textbf{1}_{\eta B}\Vert.$$
		
		Now, using convexity, we are done.
	\end{proof}
	
	\textbf{Acknowledgments:} The first author thanks the University of Murcia for partially supporting of his research stay in the University of Illinois at Urbana-Champaign in September 2017, where this paper began.
	 

\begin{thebibliography}{1}
		\bibitem{AA1}
		\textsc{F. Albiac, J. L. Ansorena}, \textit{Characterization of 1-quasi greedy bases},  J. Approx. Theory, \textbf{201}, 7-12 (2016).
		\bibitem{AA2}\textsc{F. Albiac, J. L. Ansorena}, \textit{Characterization of 1-almost greedy bases}, Rev. Mat. Complut., \textbf{30}(1), 13-24 (2017).	
		\bibitem{AK16}
		\textsc{F. Albiac and N. Kalton},
		{\it Topics in Banach Space Theory},
		second edition, ISBN 978-3-319-31555-3 (2016).	
		\bibitem{AW}\textsc{F. Albiac, P. Wojtaszczyk}, \textit{Characterization of 1-greedy bases}, J. Approx. Theory, \textbf{138} (2006), 65-86.
		\bibitem{BB} \textsc{P. M. Bern\'a, \'O. Blasco}, \textit{Characterization of greedy bases in Banach spaces}, J. Approx. Theory, \textbf{205} (2017), 28-39.
		\bibitem{BB2} \textsc{P. M. Bern\'a, \'O. Blasco}, \textit{The best m-term approximation with respect to polynomials with constant coefficients}, Anal. Math., \textbf{43} (2) (2017), 119-132.
		\bibitem{BBG}\textsc{P. M. Bern\'a, \'O. Blasco, G. Garrig\'os}, \textit{Lebesgue inequalities for the greedy algorithm in general bases}, Rev. Mat. Complut. 30 (2017), 369-392. 
		\bibitem{CDH}\textsc{A. Cohen, R. A. DeVore, R. Hochmuth}, \textit{Restricted nonlinear approximation}, Constr. Approx., 16 (2000), 85-113.
		
		%
		\bibitem{DKK} \textsc{S. J. Dilworth, N. J. Kalton, D. Kutzarova}, \textit {On the existence of almost greedy bases in Banach spaces}, Studia Math. {\bf 159} (2003), 67--101. 
		
			\bibitem{DKO} \textsc{S. J. Dilworth, D. Kutzarova, T. Oikhberg}, \textit {Lebesgue constants for the weak greedy algorithm}, Rev. Mat. Complut. {\bf 28}(2), 393--409 (2015). 
			
		\bibitem{DKOSS}\textsc{S. J. Dilworth, D. Kutzarova, E. Odell, T. Schlumprecht, A. Zs\'ak}, \textit{Renorming spaces with greedy bases}, J. Approx. Theory \textbf{188} (2014), 39-56.	
		\bibitem{DKKT}\textsc{S. J. Dilworth, N. J. Kalton, D. Kutzarova, V. N. Temlyakov}, \textit{The thresholding greedy algorithm, greedy bases, and duality}, Constr. Approx. \textbf{19} (2003), no.4, 575-597.
		\bibitem{DKTW}\textsc{S. J. Dilworth, D. Kutzarova, V. N. Temlyakov, B. Wallis}, \textit{Weight-Almost greedy bases}. (Preprint) \url{http://arxiv.org/abs/1803.02932v1}.
		
		\bibitem{GHO} \textsc{G. Garrig\'os, E. Hern\'andez, T. Oikhberg}, \textit{Lebesgue-type inequalities for quasi-greedy bases}, Constr. Approx. \textbf{38} (2013), 447-470.	
		
		\bibitem{GPT} \textsc{G. Kerkyacharian, D. Picard, V. N. Temlyakov}, \textit{Some inequalities for the
			tensor product of greedy bases and weight-greedy bases}, East J. Approx., \textbf{12} (2006), 103-118.
		
		\bibitem{KT}\textsc{S. V. Konyagin, V. N. Temlyakov}, \textit{A remark on greedy approximation in Banach spaces}, East J. Approx. \textbf{5} (1999), 365-379.
		\bibitem{Tem} \textsc{V. N. Temlyakov}, \textit{Greedy approximation}, Cambridge Monographs on Applied and Computational Mathematics, vol.20, Cambridge University Press, Cambridge, 2011.
		\bibitem{Woj} \textsc{P. Wojtaszczyk}, \textit{Greedy algorithm for general biorthogonal systems}, J. Approx.Theory \textbf{107} (2) (2000), 293-314.
		\bibitem{Woj1} \textsc{P. Wojtaszczyk}, \textit{Greedy type bases in Banach spaces}, Constructive theory of functions, 136-155, DARBA, Sofia, 2003.
		\bibitem{Woo75}
		\textsc{J. Y. T. Woo}, \textit{On a class of universal modular sequence spaces},
		{Israel Journal of Mathematics} \textbf{20} (1975), 193-215.
	\end{thebibliography}
\end{document}